\documentclass[12pt,reqno,a4paper]{amsart}
\usepackage{
    amsmath,  amsfonts, amssymb,  amsthm,   amscd,
    gensymb,  graphicx, comment,  etoolbox, url,
    booktabs, stackrel, mathtools,enumitem, mathdots,  microtype, lmodern,    mathrsfs, graphicx, tikz,  longtable,tabularx, float, tikz, pst-node, tikz-cd, multirow, tabularx, amscd,  bm, array, makecell, diagbox, booktabs,ragged2e, caption, subcaption }
\usepackage{makecell,slashbox}
\usepackage{xcolor}
\usepackage[all]{xy}
\usepackage{graphicx}
\usepackage[pagebackref=true, colorlinks=true, linkcolor=blue, citecolor=blue, urlcolor=blue, breaklinks=true]{hyperref}

\usepackage[utf8]{inputenc}
\usepackage{microtype, fullpage, wrapfig,textcomp,mathrsfs,csquotes,fbb}
\usepackage[colorlinks=true, linkcolor=blue, citecolor=blue, urlcolor=blue, breaklinks=true]{hyperref}
\usepackage[capitalise]{cleveref}
\usepackage{todonotes}
\usetikzlibrary{positioning}
\usetikzlibrary{shapes,arrows.meta,calc}
\usetikzlibrary{arrows}

\theoremstyle{definition}
\newtheorem{theorem}{Theorem}[section]

\newtheorem{corollary}[theorem] {Corollary}
\newtheorem{definition}[theorem]{Definition}
\newtheorem{example}[theorem]{Example}
\newtheorem{lemma}[theorem]{Lemma}

\newtheorem{proposition}[theorem]{Proposition}
\newtheorem{remark}[theorem]{Remark}

\newtheorem*{theorem*}{Theorem}
\newtheorem*{proposition*}{Proposition}

\newcommand\Q{\mathbb{Q}}
\newcommand\R{\mathbb{R}}
\newcommand\Z{\mathbb{Z}}
\newcommand\C{\mathbb{C}}

\newcommand{\TC}{\mathrm{TC}}

\newcommand{\ct}{\mathrm{cat}}

\newcommand{\sct}{\mathrm{secat}}

\DeclareMathOperator{\bary}{bar}
\DeclareMathOperator{\supp}{supp}
\DeclareMathOperator{\dsecat}{dsecat}
\DeclareMathOperator{\dcat}{dcat}
\newcommand{\dD}{\operatorname{dD}}

\newcommand{\SP}{\mathrm{SP}}
\newcommand{\Sig}{\Sigma}

\newcommand{\dTC}{\operatorname{dTC}}
\newcommand{\cl}{\operatorname{cl}}
\newcommand{\id}{\mathrm{id}}
\newcommand{\diag}{\Delta}
\newcommand{\sd}{\delta}
\newcommand{\pr}{\pi}
\newcommand{\iso}{\xrightarrow{\ \cong\ }}
\DeclareMathOperator{\im}{im}
\newcolumntype{x}[1]{>{\centering\arraybackslash}p{#1}}

\begin{document}
\title[]{New Bounds on Distributional Sectional Category and Applications to Distributional Homotopic Distance}

\author[N. Daundkar]{Navnath Daundkar}
\address{Department of Mathematics, Indian Institute of Technology Madras, Chennai, India.}
\email{navnath@iitm.ac.in}
\author[]{J.M. Garc\'ia-Calcines}
\address{Departamento de Matem\'aticas, Estad\'istica e Investigaci\'on Operativa, Universidad de La Laguna, Avenida Astrof\'isico Francisco S\'anchez S/N, 38200 La Laguna, Spain.}
\email{jmgarcal@ull.edu.es}

\thanks{}

\begin{abstract} 
In this paper, we establish several new bounds for the distributional sectional category ($\mathrm{dsecat}$). We first prove Jauhari's conjecture, thereby establishing a cohomological lower bound for $\mathrm{dsecat}$ with arbitrary coefficients. We then obtain an analogous lower bound with rational coefficients by constructing a natural graded $\mathbb{Q}$-linear section
of the homomorphism induced in cohomology by the symmetric diagonal.
As applications, we provide several computations of the distributional sectional category. We also establish multiplicative product and composition inequalities for $\mathrm{dsecat}$. Finally, we apply these results to distributional homotopic distance. A sequential version of this invariant was recently introduced by Jauhari and Oprea; we give an equivalent formulation in terms of distributed homotopies and derive new fibration and multiplicative triangle inequalities.
\end{abstract}

\keywords{homotopic distance,  distributional Lusternik-Schnirelmann category, distributional topological complexity, distributional sectional category}
\subjclass[2020]{55M30, 55S40, 55R05, 55S15}
\maketitle

\section{Introduction}
The sectional category of a map was introduced by Schwarz for fibrations \cite{Sva}, extended to arbitrary maps by Berstein and Ganea in~\cite{B-G}, and subsequently developed by James in~\cite{James}. It provides a common framework for several important numerical homotopy invariants.
Classical invariants such as the Lusternik--Schnirelmann category \cite{L-S-cat} and Farber's topological complexity \cite{F} arise naturally as special cases of sectional category.
A distributional approach to sectional category was introduced implicitly by Dranishnikov and Jauhari \cite{D-J} and developed by Jauhari \cite{Jau1}. Independently, Knudsen and Weinberger introduced the closely related notion of \emph{analog sectional category} \cite{K-W}. Both approaches use continuous
assignments of finitely supported probability measures on the fibres of a fibration, but employ different topologies. In the distributional setting, these measures are equipped with the L\'evy--Prokhorov metric.
This gives rise to distributional versions of the
Lusternik--Schnirelmann category, topological complexity, and their sequential analogues.

A central problem in the study of distributional sectional category is to develop effective tools for its computation. Cohomological lower bounds have played a fundamental role in the study of sectional category and its variants; see, for example, \cite{B-G,F,James}. Dranishnikov and Jauhari established cohomological lower bounds for the distributional
Lusternik--Schnirelmann category and the distributional topological complexity in \cite{D-J}, and Jauhari conjectured in \cite[Conjecture~5.9]{Jau1} that these arise as special cases of a more general cohomological lower bound for distributional sectional category. Another natural direction is to understand the behavior of $\dsecat$ under products and compositions, and its interaction with fibrations, in analogy with the corresponding properties of classical sectional category.

The primary objective of this paper is to address these questions. We first prove Jauhari's conjecture by establishing a cohomological lower bound for $\dsecat$ with arbitrary coefficients.  We then obtain a more directly applicable rational lower bound by constructing a natural splitting of
the homomorphism induced on cohomology by the symmetric diagonal.
This yields an exact kernel identity and, consequently, a lower bound in terms of the rational cup-length of $\ker p^*$. As applications, we recover the cohomological lower bound for the distributional Lusternik--Schnirelmann category of \cite[Theorem~4.7]{D-J} and, with rational coefficients, the zero-divisor cup-length bound for distributional topological complexity of
\cite[Corollary~4.13]{D-J}, and provide several new computations. We further establish multiplicative product and composition inequalities for $\dsecat$.

Finally, following the philosophy of relating sectional category to numerical invariants of maps, we study distributional homotopic distance. A sequential version of this invariant was recently introduced by Jauhari and Oprea \cite{O-J} via a pullback construction, while our formulation in terms of distributed homotopies was obtained independently. We show that, in the two-map case, the two formulations agree. This intrinsic viewpoint allows the structural results developed
for $\dsecat$ to be applied directly, yielding cohomological lower bounds, a fibration inequality, and a multiplicative triangle inequality for $\dD$.

The paper is organized as follows. Throughout, $\dsecat(p)$ denotes the distributional sectional category of a map $p\colon E\to B$, and $\dD(f,g)$ denotes the distributional homotopic distance between maps $f,g\colon X\to Y$. We write $\SP^m(X)$ for the $m$-fold symmetric product of a space $X$, and $\SP^m(p)\colon\SP^m(E)\to\SP^m(B)$ for the map induced by $p$. We denote by $\delta_m^X\colon X\to\SP^m(X)$, $x\mapsto[x,\ldots,x]$, the symmetric diagonal, obtained by composing the Cartesian diagonal $X\to X^m$ with the quotient map $X^m\to\SP^m(X)$. Finally, $\cl_{\Q}(-)$ denotes cup-length with rational coefficients. When arbitrary coefficients are considered, $R$ denotes a commutative
ring with unity.

In Section~\ref{sec: prelim}, we recall the necessary preliminaries concerning $\dsecat$ and fix notation.
Section~\ref{sec:lower bound in arbitrary coefficients} establishes a cohomological lower bound for $\dsecat$ with arbitrary coefficients.
This result recovers the lower bound for $\dcat$ of
\cite[Theorem~4.7]{D-J}; see Remark~\ref{rem:recover DJ} for the comparison with the bound for $\dTC$. In particular, Theorem~\ref{cohomological lower bound on dsecat} gives the following result.

\begin{theorem*}
Let $p\colon E\to B$ be a fibration, where $E$ is a metric space, and
let $n\geq 1$. Suppose that
$\beta_i\in H^{m_i}(\SP^{n!}(B);R)$, $1\leq i\leq n$, satisfy
$\SP^{n!}(p)^*(\beta_i)=0$. Let
$\alpha_i=(\delta_{n!}^B)^*(\beta_i)$ and assume that
$\alpha_1\smile\cdots\smile\alpha_n\neq 0$. Then
\[
\dsecat(p)\geq n.
\]
\end{theorem*}

In Section~\ref{sec: cohomological lower bound with rational coeff}, we refine this bound in the rational setting by exploiting structural properties of symmetric products. A key tool is a natural splitting in cohomology with rational coefficients.

\begin{theorem*}
For every $m\geq 1$ and every space $X$ of locally finite simplicial homotopy type, there exists a natural graded $\Q$-linear map $s^X_m\colon H^*(X;\Q)\to H^*(\SP^m(X);\Q)$ such that
$(\delta^X_m)^*\circ s^X_m=\id_{H^*(X;\Q)}$. Consequently, for every continuous map $p\colon E\to B$ between spaces of locally finite simplicial homotopy type,
\[
(\delta^B_m)^*\bigl(\ker\SP^m(p)^*\bigr)=\ker p^*,
\]
and the restriction
$(\delta^B_m)^*|_{\ker\SP^m(p)^*}\colon
\ker\SP^m(p)^*\to\ker p^*$
is a split epimorphism with section $s^B_m|_{\ker p^*}$.
If moreover $p$ is a fibration between path-connected spaces of
locally finite CW type, with $E$ metric, then
\[
\dsecat(p)\geq\cl_\Q(\ker p^*).
\]
\end{theorem*}

Sections~\ref{sec: multiplicative prod ineq} and
\ref{sec: composition inequality} establish multiplicative product and composition inequalities for $\dsecat$, respectively.

\begin{theorem*}
Let $p\colon E\to B$ and $p'\colon E'\to B'$ be fibrations, where $E$ and $E'$ are separable metric spaces. Then
\begin{equation}\label{eq: product inequality for dsecat}
 \dsecat(p\times p')
\leq
(\dsecat(p)+1)(\dsecat(p')+1)-1 .   
\end{equation}
\end{theorem*}

\begin{theorem*}
Let $W\xrightarrow{\theta}V\xrightarrow{\rho}E$
be fibrations between separable metric spaces. Then
\[
 \dsecat(\rho\circ\theta)+1
 \leq
 \bigl(\dsecat(\rho)+1\bigr)
 \bigl(\dsecat(\theta)+1\bigr).
\]
\end{theorem*}

In Section~\ref{sec: distributional homotopic distance},
we define the \emph{distributional homotopic distance}
between maps $f,g\colon X\to Y$, denoted by $\dD(f,g)$,
as the least nonnegative integer $n$ for which there exists an $(n+1)$-distributed homotopy between them.

As an application of the cohomological lower bound for $\dsecat$, we obtain the following lower bound for $\dD$ in
Proposition~\ref{cohomological lb2}. Let $\pi_Y\colon PY\to Y\times Y$ denote the free path fibration.

\begin{proposition*}
Let $f,g\colon X\to Y$ be maps between metric spaces, and let
$n\geq 1$. Suppose there exist classes
$z_i\in H^{m_i}(\SP^{n!}(Y\times Y);R)$, $1\leq i\leq n$, such that
$\SP^{n!}(\pi_Y)^*(z_i)=0$ for every $i$, and
\[
(\delta_{n!}^X)^*\SP^{n!}(f,g)^*(z_1)\smile\cdots\smile
(\delta_{n!}^X)^*\SP^{n!}(f,g)^*(z_n)\neq 0.
\]
Then
\[
\dD(f,g)\geq n.
\]
\end{proposition*}

In Section~\ref{sec: fibration inequality}, the composition inequality
for $\dsecat$ yields the following fibration inequality for $\dD$
(Theorem~\ref{thm: fibration inequality}).

\begin{theorem*}
Let $p\colon E\to B$ and $p'\colon E'\to B'$ be fibrations between
separable metric spaces, with $B$ path-connected. Let
$f,g\colon E\to E'$ be fibre-preserving maps covering
$\bar f,\bar g\colon B\to B'$. Fix $b_0\in B$ such that
$\bar f(b_0)=\bar g(b_0)=b'_0$, and let
$F=p^{-1}(b_0)$ and $F'=(p')^{-1}(b'_0)$. Then the restrictions
$f_0,g_0\colon F\to F'$ satisfy
\[
\dD(f,g)+1
\leq
\bigl(\dD(f_0,g_0)+1\bigr)
\bigl(\dcat(B)+1\bigr).
\]
\end{theorem*}

Using the multiplicative product inequality for $\dsecat$,
Section~\ref{sec: multiplicative triangle inequality} establishes the
following multiplicative triangle inequality.

\begin{theorem*}
Let $f,g,h\colon X\to Y$ be maps between separable metric spaces. Then
\[
\dD(f,g)+1
\leq
\bigl(\dD(f,h)+1\bigr)\bigl(\dD(h,g)+1\bigr).
\]
\end{theorem*}

\section{Preliminaries}\label{sec: prelim}
We begin by fixing notation and recalling the basic notions on spaces of measures that will be used throughout the paper.

For a metric space $Z$ and an integer $k \geq 1$, let $\mathcal{B}_k(Z)$ denote the subspace of the space $\mathcal{B}(Z)$ of Borel probability measures on $Z$ consisting of those measures whose support has cardinality at most $k$. We equip $\mathcal{B}(Z)$ with the L\'evy--Prokhorov metric
\[
d_{\mathrm{LP}}(\mu,\nu)
=
\inf\left\{
\varepsilon>0 :
\mu(A)\leq \nu(A^\varepsilon)+\varepsilon
\ \text{and}\
\nu(A)\leq \mu(A^\varepsilon)+\varepsilon
\ \text{for every Borel set }A\subseteq Z
\right\},
\]
where
\[
A^\varepsilon
=
\{z\in Z : d(z,A)<\varepsilon\}.
\]
The space $\mathcal{B}_k(Z)$ is endowed with the corresponding subspace topology (see \cite{P} for more details).

An element $\mu \in \mathcal{B}_k(Z)$ can be written as
$$
\mu = \sum_{i \in K} a_i\, \delta_{z_i},
$$
where $K$ is an index set with $|K| \leq k$, $z_i \in Z$, and the coefficients satisfy $a_i \geq 0$ and $\sum_{i \in K} a_i = 1$. It will often be convenient to identify each Dirac delta $\delta_{z_i}$ with the corresponding point $z_i \in Z$, and to regard $\mu$ as a formal unordered convex combination
$$
\mu = \sum_{i \in K} a_i\, z_i;
$$
we shall use one notation or the other depending on the setting.

The \emph{support} of $\mu$ is
$$
\operatorname{supp}(\mu) = \{z_i \in Z : a_i > 0,\ i \in K\}.
$$

\begin{remark}
Suppose $Z$ is a separable metric space. Then the L\'evy--Prokhorov topology on $\mathcal{B}(Z)$ coincides with the \emph{weak topology}, that is, the topology in which $\mu_n \to \mu$ if and only if
$$
\int_Z f\, d\mu_n \;\longrightarrow\; \int_Z f\, d\mu
$$
for every $f$ in the space $C_b(Z)$ of bounded continuous real-valued functions on $Z$ (see \cite{P}, \cite{Billingsley1999}).

Since $\mathcal{B}_k(Z) \subset \mathcal{B}(Z)$ carries the subspace topology, this identification restricts accordingly: for $\mu_n, \mu \in \mathcal{B}_k(Z)$,
$$
\mu_n \to \mu \quad \text{if and only if} \quad \int_Z f\, d\mu_n \;\longrightarrow\; \int_Z f\, d\mu \quad \text{for every } f \in C_b(Z).
$$
\end{remark}

\bigskip
The distributional sectional category was first introduced by Dranishnikov and Jauhari in~\cite{D-J}, and was subsequently developed in detail by Jauhari in~\cite{Jau1}. In what follows, we recall the definition of the distributional sectional category and discuss its basic properties.

Let $p\colon E \to B$ be a map, where $E$ is a metric space and $B$ is a topological space.
For $k \geq 1$, define
\[
\begin{aligned}
E_k(p)
  &= \bigcup_{x \in B} \mathcal{B}_k(p^{-1}(x))  \\
  &= \bigcup_{x \in B} \{\mu \in \mathcal{B}_k(E) : \operatorname{supp}(\mu) \subseteq p^{-1}(x)\},
\end{aligned}
\]
and define the map
\[
\mathcal{B}_k(p)\colon  E_k(p) \longrightarrow B, 
\qquad 
\mathcal{B}_k(p)(\mu) = x 
\ \text{whenever } \mu \in \mathcal{B}_k(p^{-1}(x)).
\]
It was shown in \cite[Proposition 5.1]{D-J} that, if \(p\) is a fibration, then $\mathcal{B}_k(p)$ is also a fibration.

\begin{definition}[{\cite[Definition 5.2]{Jau1}}]
Let \(p \colon E \to B\) be a fibration, where \(E\) is a metric space. 
The \emph{distributional sectional category} of \(p\), denoted by \(\mathrm{dsecat}(p)\), is the least integer \(n \geq 0\) such that the fibration
$\mathcal{B}_{n+1}(p) \colon E_{n+1}(p) \longrightarrow B$
admits a (continuous) section. If no such integer exists, we set \(\mathrm{dsecat}(p)=\infty\).
\end{definition}

When \(B\) is, in addition, paracompact, this construction provides a lower bound for the usual sectional category of the fibration \(p\):
\[
\mathrm{dsecat}(p)\;\leq\;\mathrm{secat}(p).
\]

An appealing feature of the construction \(\mathcal{B}_k(p)\colon E_k(p)\to B\) is its functorial behavior for each integer \(k\ge 1\).  
Indeed, given a commutative square,
\[
\xymatrix{
{E}  \ar[rr]^f \ar[d]_p & & E' \ar[d]^{p'} \\
 {B} \ar[rr]_g & & {B'},
}
\]
where \(p\) and \(p'\) are fibrations with $E$ and $E'$ metric spaces, there is an induced commutative square
\[
\xymatrix{
{E_k(p)}  \ar[rr]^{f^k_*} \ar[d]_{\mathcal{B}_k(p)} & & {E_k(p')} \ar[d]^{\mathcal{B}_k(p')} \\
 {B} \ar[rr]_g & & {B'},}
\]
\noindent where $f^k_*$ is defined naturally.  
Moreover, this assignment is functorial: it preserves compositions and identities, that is,
\[
(g\circ f)^k_*=g^k_*\circ f^k_*
\qquad\text{and}\qquad
(\mathrm{id}_E)^k_*=\mathrm{id}_{E_k(p)}.
\]

Distributional sectional category shares most of the pleasant formal properties of the classical sectional category.

\begin{proposition}[{\cite{Dau-J}}]\label{pullbacks}
Consider the following pullback
\[
\xymatrix{
{E}  \ar[rr]^f \ar[d]_p & & E' \ar[d]^{p'} \\
 {B} \ar[rr]_g & & {B'},
}
\]
where \(p\) and \(p'\) are fibrations with $E$ and $E'$ metric spaces. Then, 
\[
\xymatrix{
{E_k(p)}  \ar[rr]^{f^k_*} \ar[d]_{\mathcal{B}_k(p)} & & {E_k(p')} \ar[d]^{\mathcal{B}_k(p')} \\
 {B} \ar[rr]_g & & {B'}}
\]
\noindent is also a pullback for each $k\geq 1$.     
\end{proposition}

As a consequence of Proposition \ref{pullbacks} above, we have:

\begin{proposition}[{\cite[Proposition 3.4]{Dau-J}}] \label{segundo}
Let $p\colon E\to B$ be a fibration and \(f\colon B'\to B\) any map, where $E$ and $B'$ are metric spaces.  
If \(p'\colon E'\to B'\) denotes the pullback of \(p\) along \(f\),
then
$\mathrm{dsecat}(p') \leq \mathrm{dsecat}(p).$
\end{proposition}

We also have homotopy invariance in the following sense.

\begin{proposition}[{\cite[Proposition 5.3]{Jau1}\label{primero}}]
Let $p\colon E\to B$ and \(p'\colon E'\to B'\) be fibrations, with $E$ and $E'$ metric spaces, fitting into a homotopy commutative diagram
\[
\xymatrix{
E \ar[rr]^{\alpha }_{\simeq } \ar[d]_{p} & & E' \ar[d]^{p'} \\
B \ar[rr]_{\beta }^{\simeq } & & B',
}
\]
where both \(\alpha\) and \(\beta\) are homotopy equivalences.  
Then $\mathrm{dsecat}(p)=\mathrm{dsecat}(p').$
\end{proposition}
This result allows us to extend the definition of \(\mathrm{dsecat}(p)\) to maps
\(p\colon E\to B\) between metric spaces that are not fibrations. Indeed, in such a
case we can choose a \emph{substitutional fibration} of \(p\), that is, a
factorization of the form
\[
\xymatrix{
E \ar[rr]^p \ar[dr]_{h}^{\simeq } & & B \\
 & \overline{E}, \ar[ur]_{\overline{p}} &
}
\]
where \(h\) is a homotopy equivalence, \(\overline{p}\) is a fibration and
\(\overline{E}\) is a metric space; for instance, one can take the standard
mapping path space
\(\overline{E}=\{(e,\gamma)\in E\times B^I:\ p(e)=\gamma(0)\}\), with
\(h(e)=(e,c_{p(e)})\) and \(\overline{p}(e,\gamma)=\gamma(1)\). 
We then define
\[
\mathrm{dsecat}(p) := \mathrm{dsecat}(\overline{p}).
\]
\begin{remark}
By Proposition~\ref{primero}, this definition is independent of the particular
choice of substitutional fibration with metric total space.
Moreover, the statement of Proposition~\ref{primero} can be extended to arbitrary
maps \(p\) and \(p'\) between metric spaces that need not be fibrations.
\end{remark}

\begin{remark}\label{dsecat hp inequality}
Again, it is not difficult to verify that Proposition~\ref{segundo} remains valid
for homotopy pullbacks, even when \(p\) and \(p'\) are maps between metric spaces
that are not fibrations.
\end{remark}

\bigskip
A basic monotonicity property with respect to morphisms of fibrations also holds.

\begin{proposition}[{\cite[Proposition 3.5]{Dau-J}}]\label{prop: dsecat monotonicity}
Let $p\colon E\to B$ and $p'\colon E'\to B$ be fibrations, with $E$ and $E'$ metric
spaces, fitting into the homotopy commutative triangle
$$\xymatrix{
{E} \ar[rr] \ar[dr]_p & & {E'} \ar[dl]^{p'} \\
 & {B}. & }$$
Then, $\mathrm{dsecat}(p')\leq \mathrm{dsecat}(p).$
\end{proposition}

Let $B$ be a path-connected metric space with a fixed basepoint $b_0\in B$, and
let $P_0B=\{\beta\colon I\to B:\ \beta(1)=b_0\}$ be the based path space, endowed
with the compact--open topology; since $B$ is metric and $I$ is compact, $P_0B$ is
metrizable. Denote by $p^B\colon P_0B\to B$, $p^B(\beta)=\beta(0)$, the based path
fibration. Recall from \cite[Definition ~3.1]{D-J} that the
\emph{distributional LS-category} of $B$, denoted $\mathrm{dcat}(B)$, is the least
integer $n\geq 0$ for which there exists a continuous map
$H\colon B\to \mathcal{B}_{n+1}(P_0B)$ such that
$H(b)\in\mathcal{B}_{n+1}\bigl(P(b,b_0)\bigr)$ for every $b\in B$, where
$P(b,b_0)$ denotes the space of paths in $B$ from $b$ to $b_0$.
Such a map $H$ is exactly a section of $\mathcal{B}_{n+1}(p^B)$, and therefore
$\mathrm{dcat}(B)=\mathrm{dsecat}(p^B)$
(see also \cite[Proposition 5.7]{Jau1}).

\begin{corollary}\label{cor-dsecat-dcat}
Let $p\colon E\to B$ be a fibration, where $E$ is a nonempty metric space and $B$
is a path-connected metric space. Then $\mathrm{dsecat}(p)\leq \mathrm{dcat}(B).$
Moreover, if $E$ is contractible, then $$\mathrm{dsecat}(p)= \mathrm{dcat}(B).$$
\end{corollary}

\begin{proof}
Since $E\neq\emptyset$, $B$ is path-connected and $p$ is a fibration, $p$ is
surjective; choose $e_0\in p^{-1}(b_0)$, where $b_0$ is the basepoint used to
define $P_0B$. Fix a continuous lifting function
$\Gamma\colon E\times_B PB\to PE$ for $p$, and define
$\omega\colon P_0B\to E$ by $\omega(\beta)=\Gamma(e_0,\bar\beta)(1)$, where
$\bar\beta(t)=\beta(1-t)$. Then $p(\omega(\beta))=\bar\beta(1)=\beta(0)$, so
$\omega$ fits into the commutative diagram
$$\xymatrix{
{P_0B} \ar[rr]^{\omega } \ar[dr]_{p^B} & & {E} \ar[dl]^p \\
& {B}. &  }$$
Therefore, by Proposition~\ref{prop: dsecat monotonicity},
$\mathrm{dsecat}(p)\leq \mathrm{dsecat}(p^B)=\mathrm{dcat}(B).$

Moreover, if $E$ is contractible, then $\omega$ is a homotopy equivalence, since
$P_0B$ is always contractible and any map between contractible spaces is a
homotopy equivalence. As the diagram above commutes strictly,
Proposition~\ref{primero} applies to $\alpha=\omega$ and $\beta=\mathrm{id}_B$,
giving $\mathrm{dsecat}(p)=\mathrm{dsecat}(p^B)=\mathrm{dcat}(B)$.
\end{proof}

We use Corollary~\ref{cor-dsecat-dcat} and a result that $\dcat(\R \mathrm{P}^n)=1$ of Dranishnikov and Jauhari \cite[Proposition 6.2]{D-J} to obtain an example of a fibration $p$ for which the strict inequality $\dsecat(p)< \sct(p)$ holds.
\begin{example}
\normalfont{
Consider the double covering
$p\colon S^n \longrightarrow \mathbb{R}\mathrm{P}^n$, $n\ge 2$.
Let $x\in H^1(\mathbb{R}\mathrm{P}^n;\Z_2)$ denote the canonical generator. Since $p^*(x)=0$,
we have $x\in\ker p^*$. Moreover,
\[
x^n\neq 0\in H^n(\mathbb{R}\mathrm{P}^n;\Z_2),
\]
so the classical cohomological lower bound for sectional category yields
\[
\sct(p)\ge \operatorname{cup\text{-}length}(\ker p^*)\ge n.
\]
Since $\sct(p)\le \ct(\mathbb{R}\mathrm{P}^n)=n$,
it follows that $\sct(p)=n$.

On the other hand, Corollary~\ref{cor-dsecat-dcat} implies that $\dsecat(p)\le \dcat(\mathbb{R}\mathrm{P}^n)=1$.
Since the covering $p$ admits no section, we have $\dsecat(p)\neq 0$. Hence $\dsecat(p)=1$.
Therefore,
\[
\dsecat(p)=1<n=\sct(p),
\]
showing that the gap $\sct(p)-\dsecat(p)$
can be made arbitrarily large.}
\end{example}

We give another example of a fibration for which $\dsecat<\sct$. 

\begin{example}
Consider $S^{2n-1}\subseteq \C^n$, with $n\geq 2$,
and let $\Z_m$, $m\geq 2$, be the group of $m$-th roots
of unity in $\C$. The group $\Z_m$ acts on $S^{2n-1}$ via componentwise complex multiplication. 
The corresponding orbit space is the lens space $L^{2n-1}_m:=\frac{S^{2n-1}}{\Z_m}$. Several estimates for $\dcat$ and $\dTC$ have been recently  obtained by Dranishnikov in \cite[Theorem B and C]{dranishnikov2025distributional}. Consider the quotient map $p\colon S^{2n-1}\to L^{2n-1}_m$. 
The cohomology ring of a lens space is described in
\cite[Example 3E.2]{hatcher} as follows:
\begin{equation}\label{eq:coho-ring-lens-spaces}
H^*(L_m^{2n-1};\mathbb{Z}_m)
\cong
\frac{\mathbb{Z}_m[x,y]}{\langle y^{n},x^2-ay\rangle},
\end{equation}
where $|x|=1$, $|y|=2$, and
\[
a=
\begin{cases}
0, & \text{if } m \text{ is odd},\\
m/2, & \text{if } m \text{ is even}.
\end{cases}
\] We also have $xy^{n-1}\neq 0$ in fact it generates the top dimensional cohomology. Since $S^{2n-1}$ has trivial cohomology
in degrees $1$ and $2$, we have $x,y\in\ker p^*$. Consequently, $xy^{n-1}\neq 0$ implies that $\sct(p)\geq n$. 

For $\dsecat$, it follows from Corollary~\ref{cor-dsecat-dcat} that $\dsecat(p)\leq \dcat(L^{2n-1}_m)$. \cite[Theorem C]{dranishnikov2025distributional} provides several examples of lens spaces for which $\dcat(L^{2n-1}_m)<n$. For example, taking $n=p$, where $p$ is a prime number, Theorem C yields
$\dcat(L^{2p-1}_p)=p-1<p$. This shows we have many examples for which $\dsecat(p)<\sct(p)$.
\end{example}

\section{Cohomological lower bound using arbitrary coefficients}\label{sec:lower bound in arbitrary coefficients}
In this section, we prove Jauhari's conjecture \cite[Conjecture 5.9]{Jau1}, which concerns a cohomological lower bound for the distributional sectional category. Before proceeding, we prove an auxiliary lemma.

Recall that the $k$-th \emph{symmetric product} of a topological space $Z$, denoted
$\SP^k(Z)$, is the orbit space of $Z^k$ under the permutation action of the
symmetric group $\Sigma_k$. Its elements are orbits $[z_1,\dots,z_k]$ of
points $(z_1,\dots,z_k)\in Z^k$, and any such orbit can be regarded as a
formal sum $\sum k_i z_i$ with $\sum k_i=k$, subject to the identification
$(n+\ell)z=nz+\ell z$.

Now, let $Z$ be a metric space and $\mathcal{B}_n^{(i)}(Z)
=
\left\{
\mu\in\mathcal{B}_n(Z) :
\left|\operatorname{supp}(\mu)\right|=i
\right\}$, for $i\leq n$.
Whenever \(i\) divides \(N\), we define
\[
r_{i,N}\colon
\mathcal{B}_n^{(i)}(Z)
\longrightarrow
\SP^N(Z)
\quad  \text{ by } \quad
r_{i,N}(\mu)
:=
\sum_{z\in\operatorname{supp}(\mu)}
\frac{N}{i}\,z.
\]

\begin{lemma}\label{continuous}
 The map \(r_{i,N}\) is continuous.
\end{lemma}
\begin{proof}
For $a_j>0$, we fix $\mu = \sum_{j=1}^{i}a_j\delta_{x_j}$,
where \(x_1,\ldots,x_i\) are pairwise distinct. Given
\(\varepsilon>0\), choose \(\eta>0\) such that
\[
\eta
<
\min\left\{
\varepsilon,\,
\min_{1\leq j\leq i}a_j,\,
\frac{\delta}{2}
\right\}, \ 
\text{ where } \ 
\delta=
\begin{cases}
\displaystyle\min_{j\neq k} d(x_j,x_k), & \text{if } i\ge 2,\\[1ex]
+\infty, & \text{if } i=1.
\end{cases}
\]
Suppose that $\nu\in\mathcal{B}_n^{(i)}(Z)$ and
$d_{\mathrm{LP}}(\mu,\nu)<\eta$.
By the defining property of the L\'evy--Prokhorov metric, applied to the
singleton \(\{x_j\}\), we obtain
\[
a_j
=
\mu(\{x_j\})
\leq
\nu\bigl(B(x_j,\eta)\bigr)+\eta.
\]
Since \(\eta<a_j\), it follows that 
$\nu\bigl(B(x_j,\eta)\bigr)>0$
for every \(j\).
The balls \(B(x_j,\eta)\) are pairwise disjoint. Since \(\nu\) has
exactly \(i\) points in its support, each of these \(i\) balls contains
exactly one support point \(y_j\), and there are no support points
outside their union. Thus, after a suitable labelling, we have 
$
\operatorname{supp}(\nu)
=
\{y_1,\ldots,y_i\}$ with  $d(x_j,y_j)<\eta$ for every $j$.

Equip $Z^N$ with the maximum metric defined as
$$
d_{\infty}
\bigl(
(z_1,\ldots,z_N),(w_1,\ldots,w_N)
\bigr)
=
\max_{1\leq \ell\leq N}d(z_\ell,w_\ell).
$$
Note that \(\Sigma_N\) acts on \(Z^N\) by isometries. Since
\(\Sigma_N\) is finite, the metric
$$
d_{\SP}
\bigl(
[z_1,\ldots,z_N],
[w_1,\ldots,w_N]
\bigr)
=
\min_{\tau\in\Sigma_N}
\max_{1\leq \ell\leq N}
d\bigl(z_\ell,w_{\tau(\ell)}\bigr)
$$
induces the usual quotient topology on
$\SP^N(Z)$.
Pairing the \(N/i\) copies of \(x_j\) with the \(N/i\) copies of \(y_j\),
we obtain
$$
d_{\SP}
\bigl(
r_{i,N}(\mu),
r_{i,N}(\nu)
\bigr)
\leq
\max_{1\leq j\leq i}d(x_j,y_j)<\eta< \varepsilon.
$$
This proves that \(r_{i,N}\) is continuous.    
\end{proof}

For a space $B$, the map
$\delta_k^B\colon B\to \SP^k(B)$ defined by
$\delta_k^B(x):=[x,x,\dots,x]=kx$ is called the \emph{symmetric diagonal}.
This map is an embedding, and we shall identify $B$ with its image
$\delta_k^B(B)\subseteq \SP^k(B)$ whenever convenient. Likewise, for a
subset $A_i\subseteq B$, we identify $A_i$ with its image
$\delta_k^B(A_i)\subseteq \SP^k(B)$.
We denote by
$j_{i,k}\colon A_i\to \SP^k(B)$
the restriction of $\delta_k^B$ to $A_i$; thus
$j_{i,k}=\delta_k^B\circ\mathrm{inc}_{A_i}$, where $\mathrm{inc}_{A_i}\colon A_i\hookrightarrow B$ is the inclusion.

\begin{proposition}\label{lem: dsecat cover}
Let $p\colon E\to B$ be a fibration, where $E$ is a metric space, and let
$n\geq 1$. Suppose that $\mathrm{dsecat}(p)<n$. Then there exist subsets
$A_1,\dots,A_n\subseteq B$ with $B=\bigcup_{i=1}^n A_i$ and maps
$\sigma_i\colon A_i\to \SP^{n!}(E)$ such that
$\SP^{n!}(p)\circ\sigma_i=j_{i,n!}$.
\end{proposition}

\begin{proof}
Suppose $\mathrm{dsecat}(p)<n$. Then there exists a section $s$ of $\mathcal{B}_n(p)\colon E_n(p)\to B$.
For each $1\leq i\leq n$, define 
$$
A_i:=\{a\in B \mid |\operatorname{supp}(s(a))|=i\}.
$$ These sets cover $B$.
Then define $\sigma_i\colon A_i\to \SP^{n!}(E)$ by 
$$\sigma_i(a):= \sum_{e\in \operatorname{supp}(s(a))}\frac{n!}{i}e.$$
We now check the continuity of $\sigma_i$. Note that the restriction of \(s\) gives a continuous map
$$
s|_{A_i}\colon
A_i
\longrightarrow
\mathcal{B}_n^{(i)}(E),
\quad 
\text{ and } \quad
\sigma_i
=
r_{i,n!}\circ s|_{A_i}.
$$
Since \(i\mid n!\), the map \(r_{i,n!}\) is defined, and it is continuous by Lemma~\ref{continuous}; therefore \(\sigma_i\) is continuous.

If $\operatorname{supp}(s(a))=\{e_1,\dots,e_i\}$,
then $$\SP^{n!}(p)\circ \sigma_i(a)=\sum_{j=1}^i\frac{n!}{i}p(e_j)=\sum_{j=1}^i\frac{n!}{i}a=n!a=j_{i,n!}(a).$$
This completes the proof.
\end{proof} 

\begin{remark}
The sets $A_i$ and the maps $\sigma_i$ adapt the constructions of
\cite[Lemmas~4.4 and~4.9]{D-J} and \cite[Lemma~4.3]{Jau1}; see also
\cite[Lemma~6.1]{O-J}. For the parametrized fibrations $\Pi^E_r$, a
cohomological lower bound of this type was obtained in
\cite[Theorem~5.1]{Dau-J}. In those settings the sets $A_i$ deform into
a diagonal, or the relevant fibration is of path type, so the hypothesis
can be imposed on the classes downstairs. For an arbitrary fibration no
such deformation is available, and we argue instead with the lifts
$\sigma_i$. Moreover, in the proof below we use tautness of
Alexander--Spanier cohomology to replace the sets $A_i$ by open
neighbourhoods before applying the relative cup product.
\end{remark}

We are now in a position to obtain the cohomological lower bound for the distributional sectional category conjectured by Jauhari in
\cite[Conjecture 5.9]{Jau1}. The corresponding cohomological lower bounds for $\dcat$ and $\dTC$ were established by Dranishnikov and Jauhari in \cite{D-J} for the based path fibration and the free path fibration, respectively. Here, we establish the corresponding result for arbitrary fibrations.

We use Alexander--Spanier cohomology \cite{Sp}
because of its tautness property for arbitrary
subspaces of metrizable spaces
\cite[Theorem~1(3)]{SpanierTautness1978}.
This property allows us to pass from the subsets
$A_i$ in Proposition~\ref{lem: dsecat cover},
which need not be open, to suitable open
neighbourhoods on which the relevant cohomology
classes vanish. The relative cup product can
then be applied to these open subsets.
We use coefficients in a commutative ring with
unity $R$ and, for brevity, omit $R$ from
the notation.

\begin{theorem}\label{cohomological lower bound on dsecat}
Let $p\colon E\to B$ be a fibration, where $E$ is a metric space, and let $n\geq 1$. Suppose $\beta_i\in H^{m_i}(\SP^{n!}(B))$, $1\leq i\leq n$, satisfy $\SP^{n!}(p)^*(\beta_i)=0$. Let $\alpha_i=(\delta_{n!}^B)^*(\beta_i)$ and assume that $\alpha_1\smile\cdots\smile\alpha_n\neq 0$. Then
$$
\mathrm{dsecat}(p)\geq n.
$$
\end{theorem}

\begin{proof}
Suppose, on the contrary, that $\dsecat(p)<n$,
and let $s\colon B\to E_n(p)$ be a section of
$\mathcal{B}_n(p)$.
Let $A_1,\dots,A_n$ and
$\sigma_i\colon A_i\to\SP^{n!}(E)$ be the sets
and maps constructed from $s$ in the proof of
Proposition~\ref{lem: dsecat cover}. Thus,
$B=\bigcup_{i=1}^n A_i$ and
$\SP^{n!}(p)\circ\sigma_i=j_{i,n!}$.

Since $\SP^{n!}(p)^*(\beta_i)=0$, we obtain
\[
\begin{aligned}
\alpha_i|_{A_i}
&=\mathrm{inc}_{A_i}^*(\delta_{n!}^B)^*(\beta_i)\\
&=j_{i,n!}^*(\beta_i)\\
&=\sigma_i^*\SP^{n!}(p)^*(\beta_i)=0.
\end{aligned}
\]

Since $\mathcal{B}_n(p)\circ s=\mathrm{id}_B$,
the map $s$ embeds $B$ into the metric space
$\mathcal{B}_n(E)$. Hence $B$ is metrizable.
Since every open subset of $B$ is metrizable,
and therefore paracompact Hausdorff, each $A_i$
is taut in $B$ with respect to Alexander--Spanier
cohomology
\cite[Theorem~1(3)]{SpanierTautness1978}.
Consequently, $\alpha_i|_{A_i}=0$ implies
that there exists an open neighbourhood $U_i$
of $A_i$ in $B$ such that $\alpha_i|_{U_i}=0$.
Since the $A_i$ cover $B$, so do the $U_i$.

By the long exact sequence of the pair $(B,U_i)$,
there exists a class
$\widetilde{\alpha}_i\in H^{m_i}(B,U_i)$
whose image in $H^{m_i}(B)$ is $\alpha_i$.
Since the $U_i$ are open, the relative cup
product is defined and gives
\[
\widetilde{\alpha}_1\smile\cdots\smile
\widetilde{\alpha}_n
\in H^m\Bigl(B,\bigcup_{i=1}^n U_i\Bigr)
=H^m(B,B)=0,
\]
where $m=m_1+\cdots+m_n$.
By naturality, its image in $H^m(B)$ is
$\alpha_1\smile\cdots\smile\alpha_n$.
Thus $\alpha_1\smile\cdots\smile\alpha_n=0$,
contrary to our hypothesis.
Therefore $\dsecat(p)\geq n$.
\end{proof}

Although stated for fibrations,
Theorem~\ref{cohomological lower bound on dsecat}
also proves Jauhari's conjecture
\cite[Conjecture~5.9]{Jau1} for arbitrary maps between metric spaces. Indeed, let $p=\overline p\circ h$ be a substitutional fibration, as described in Section~\ref{sec: prelim}.
Since $h$ is a homotopy equivalence, so is $\SP^{n!}(h)$.
Consequently,
\[
\ker\SP^{n!}(p)^*=\ker\SP^{n!}(\overline p)^*.
\]
The conclusion now follows from the theorem and the equality
$\dsecat(p)=\dsecat(\overline p)$.

\begin{remark}\label{rem:recover DJ}
Theorem~\ref{cohomological lower bound on dsecat} recovers the
cohomological lower bound for $\dcat$ of \cite[Theorem~4.7]{D-J}.
Indeed, for the based path fibration $p^X\colon P_0X\to X$, the space
$\SP^{n!}(P_0X)$ is contractible, since $\SP^{n!}$ preserves homotopy
equivalences. Hence every class of positive degree lies in
$\ker\SP^{n!}(p^X)^*$, and the hypothesis of
Theorem~\ref{cohomological lower bound on dsecat} reduces to
$\alpha_1\smile\cdots\smile\alpha_n\neq0$.

For the free path fibration $\pi_X\colon PX\to X\times X$ the situation
is different. The constant-path map $c\colon X\to PX$ is a homotopy
equivalence with $\pi_X\circ c=\Delta_X$, so
$\ker\SP^{n!}(\pi_X)^*=\ker\SP^{n!}(\Delta_X)^*$. If $\beta_i$ lies in
this kernel, then $\alpha_i=(\delta^{X\times X}_{n!})^*(\beta_i)$ is a
zero-divisor, but the converse need not hold. Thus, applied to $\pi_X$,
Theorem~\ref{cohomological lower bound on dsecat} yields
$\dTC(X)\geq n$ under a hypothesis stronger than that of
\cite[Theorem~4.12]{D-J}, where only $\Delta_X^*(\alpha_i)=0$ is
required. This weaker hypothesis suffices there because the sets $A_i$
deform into the diagonal inside $\SP^{n!}(X\times X)$, a feature
specific to path fibrations; compare \cite[Remark~4.5]{Jau1}. With
rational coefficients, however, the two approaches give the same
numerical bound. For $X$ of locally finite CW type,
Theorem~\ref{thm:dsecat} gives
$\dTC(X)=\dsecat(\pi_X)\geq\cl_\Q(\ker\pi_X^*)=\cl_\Q(\ker\Delta_X^*)$,
which is \cite[Corollary~4.13]{D-J}.
\end{remark}

\section{Cohomological lower bound using rational coefficients}\label{sec: cohomological lower bound with rational coeff}
The cohomological lower bounds for sectional category, and hence for
topological complexity and LS category, rest on the zero-divisor mechanism: a product of classes in the kernel of a
restriction homomorphism witnesses a lower bound. In the distributional
setting of Dranishnikov and Jauhari \cite{D-J}, the
relevant witnesses no longer live on the base $B$ itself but on its
symmetric products $\SP^{m}(B)$, and the criterion is phrased in terms of
the symmetric diagonal $\sd^B_m\colon B\to\SP^m(B)$. This raises a
concrete difficulty: the kernel one controls in practice is
$\ker p^*\subseteq H^*(B;\Q)$ downstairs, whereas the criterion asks for
classes in $\ker\SP^m(p)^*$ upstairs, an object that is hard to access
directly.

The purpose of this section is to bridge that gap once and for all, by
constructing a \emph{natural} section of $(\sd^B_m)^*$ and deriving from
it an exact identification of the two kernels. The construction is
elementary: it averages the pullbacks of the coordinate projections. Its
force comes entirely from naturality.

The integral and mod-$p$ cohomology of symmetric products was first
studied by Nakaoka \cite{Nakaoka1957} and Dold \cite{Dold1958}; here we
need only the rational orbit isomorphism
(Proposition~\ref{prop:orbitfin}, Theorem~\ref{thm:orbitcw}), which
follows directly from the transfer theorem.

\textbf{Relation to existing results: }
It is known that the induced homomorphism $(\sd^X_m)^*$ is surjective with rational coefficients: for $X$ a finite simplicial complex it is
\cite[Prop.~4.3]{D-J}, proved there by a different route,
via the Dold--Thom theorem and the fact that multiplication by $m$ is a
rational homology isomorphism of $\SP^\infty(X)$. Surjectivity alone
yields, for each space separately, a section of $(\sd^X_m)^*$ chosen by
an arbitrary splitting. That is not sufficient for the kernel
identity: what the argument of Section~\ref{sec:kernel} requires is a
section that is natural for every continuous map, so that the lift of a
class in $\ker p^*$ automatically lands in $\ker\SP^m(p)^*$
(Remark~\ref{rem:naturality}). Supplying such a section, by explicit
symmetrisation, is the contribution of Section~\ref{sec:splitting}.

The cohomological lower bounds for $\dcat$ and $\dTC$ of
\cite[Thms.~4.7 and~4.12]{D-J} take the classes upstairs
as given and push them down. The present paper goes the other way: it
lifts a prescribed kernel downstairs to the kernel upstairs, which is the direction needed when what one controls is $\ker p^*$.

For fibrations of path type, surjectivity alone suffices: in
\cite[Theorem~4.12]{D-J}, \cite[Theorem~4.4]{Jau1} and
\cite[Theorem~5.1 and Proposition~6.2]{Dau-J}, the hypothesis is imposed
on the classes $(\sd^B_m)^*(\beta_i)$ downstairs, and any preimage under
$(\sd^B_m)^*$ will do. For a general fibration the hypothesis of
Theorem~\ref{cohomological lower bound on dsecat} is on
$\beta_i\in\ker\SP^m(p)^*$, and this is where naturality is needed.

\textbf{Conventions :}
All cohomology is singular cohomology with coefficients in $\Q$, unless a
different coefficient field is displayed. For a space $X$ and an integer
$m\ge 1$, the symmetric group $\Sig_m$ acts on $X^m$ by
\[
\sigma\cdot(x_1,\dots,x_m)=(x_{\sigma^{-1}(1)},\dots,x_{\sigma^{-1}(m)}),
\]
the $m$th symmetric product is the orbit space $\SP^m(X)=X^m/\Sig_m$, and
\[
q_X\colon X^m\longrightarrow\SP^m(X),\qquad
q_X(x_1,\dots,x_m)=[x_1,\dots,x_m]
\]
is the orbit map. We write $\pr^X_j\colon X^m\to X$ for the $j$th
coordinate projection, $\diag_X\colon X\to X^m$ for the Cartesian
diagonal, and
\[
\sd^X_m=q_X\circ\diag_X\colon X\longrightarrow\SP^m(X),\qquad
\sd^X_m(x)=[x,\dots,x],
\]
for the symmetric diagonal. Every map $f\colon X\to Y$ induces
$\SP^m(f)\colon\SP^m(X)\to\SP^m(Y)$, $[x_1,\dots,x_m]\mapsto
[f(x_1),\dots,f(x_m)]$, and the following identities are strict:
\begin{equation}\label{eq:strict}
q_Y\circ f^m=\SP^m(f)\circ q_X,\qquad
\sd^Y_m\circ f=\SP^m(f)\circ\sd^X_m.
\end{equation}
A space is of \emph{locally finite simplicial homotopy type} if it is homotopy
equivalent to the geometric realisation of a locally finite simplicial complex. This is the standing hypothesis of
subsections~\ref{sec:orbit}--\ref{sec:cuplength}. A space is of
\emph{locally finite CW type} if it has the homotopy type of a locally finite CW complex; such a space is homotopy equivalent to a metric ANR, and this stronger hypothesis is needed only in Section~\ref{sec:dsecat},
where the distributional lower bound is invoked.

With the above background, we are now ready to state our main theorem.

\begin{theorem}\label{thm:main}
For every $m\ge1$ and every space $X$ of locally finite simplicial homotopy type, there
is a natural graded $\Q$-linear map
\[
s^X_m\colon H^*(X;\Q)\longrightarrow H^*(\SP^m(X);\Q)
\]
with $(\sd^X_m)^*\circ s^X_m=\id_{H^*(X;\Q)}$. Consequently, for every continuous
map $p\colon E\to B$ between spaces of locally finite simplicial homotopy type,
\[
(\sd^B_m)^*\bigl(\ker\SP^m(p)^*\bigr)=\ker p^*,
\]
and the restriction $(\sd^B_m)^*|_{\ker\SP^m(p)^*}\colon\ker\SP^m(p)^*\to\ker p^*$ is a
split epimorphism with section $s^B_m|_{\ker p^*}$. If moreover $p$ is a
fibration between path-connected spaces of \emph{locally finite CW type}, with $E$ metric,
then
\[
\dsecat(p)\ge\cl_\Q(\ker p^*).
\]
\end{theorem}

No fibration hypothesis is used for the kernel identity $(\sd^B_m)^*\bigl(\ker\SP^m(p)^*\bigr)=\ker p^*$. The application to $\dsecat $ in Section \ref{sec:dsecat} uses in addition the homotopy invariance of $\dsecat $, in order to pass to a locally finite CW model.

\subsection{The orbit-cohomology isomorphism in the polyhedral case}
\label{sec:orbit}

Recall that the symmetric power functor  $\SP^m$ preserves homotopies and, consequently, homotopy equivalences. This will allow us to extend the results obtained in the polyhedral setting to spaces of simplicial homotopy type.

The key input is the classical transfer theorem for finite simplicial
group actions. We do not reprove it; we recall the precise statement in
the form needed here and verify that our action satisfies its
hypotheses.

\subsubsection{Regular simplicial actions and the transfer}
\label{subsec:regular}

Let a finite group $G$ act simplicially on a simplicial complex. Bredon
isolates two conditions on such an action
\cite[Ch.~III, Sec.~1, pp.~115--116]{Bredon1972}:
\begin{itemize}
\item[\textup{(A)}] for every $g\in G$ and every simplex $s$, the
  element $g$ leaves $s\cap g(s)$ pointwise fixed; equivalently, if $v$
  and $g(v)$ lie in a common simplex then $v=g(v)$;
\item[\textup{(B)}] if $(v_0,\dots,v_n)$ and $(g_0v_0,\dots,g_nv_n)$ are
  both simplices, then $g(v_i)=g_i(v_i)$ for all $i$, for some $g\in G$.
\end{itemize}
An action is \emph{regular}, and the complex a \emph{regular
$G$-complex}, when \textup{(B)} holds for every subgroup of $G$
\cite[Ch.~III, Def.~1.2]{Bredon1972}. Condition \textup{(A)} is the
``no-inversions'' requirement: a group element stabilising a simplex
setwise fixes it pointwise.

The following two facts from \cite[Ch.~III, Sec.~1]{Bredon1972} will be used in what follows. 
\begin{enumerate}
\item[Fact-1] For any simplicial $G$-complex $K$, the induced action on the barycentric
subdivision $K'$ satisfies \textup{(A)}; and if $K$ already satisfies
\textup{(A)}, then $K'$ is regular \cite[Ch.~III, Prop.~1.1]{Bredon1972}.
Hence every simplicial action becomes regular after the second
barycentric subdivision, see \cite[Ch.~III, p.~117]{Bredon1972}. 
\item[Fact-2] For a
regular $G$-complex $K$ the orbit space $K/G$ is again a simplicial
complex whose simplices are the $G$-orbits of simplices of
$K$; the orbit map is simplicial, and there is a canonical homeomorphism
$|K|/G\cong|K/G|$, see \cite[Ch.~III, Sec.~1, p.~117]{Bredon1972}.
\end{enumerate}

\begin{lemma}[Transfer for finite simplicial actions]
\label{lem:transfer}
Let $G$ be a finite group acting simplicially and regularly on a
simplicial complex $K$, and let $k$ be a field of characteristic $0$ or
prime to $|G|$. Then the orbit map $q\colon|K|\to|K|/G$ induces an
isomorphism of graded $k$-algebras
\begin{equation}\label{eq:transfer}
q^*\colon H^*(|K|/G;k)\iso H^*(|K|;k)^{G},
\end{equation}
where $H^*(|K|;k)^{G}$ is the subalgebra of $G$-invariant classes.
\end{lemma}

\begin{proof}
This is the cohomological form of the transfer theorem
\cite[Ch.~III, Thm.~2.4 and p.~121]{Bredon1972}, in the case of an empty
subcomplex. Write $\mu^*\colon H^*(|K|;k)\to H^*(|K|/G;k)$ for the
cohomological transfer. It satisfies the two identities
\begin{equation}\label{eq:transferids}
\mu^*\circ q^*=|G|\cdot\id_{H^*(|K|/G;k)},\ \text{ and } \ q^*\circ \mu^*=\sum_{g\in G}g^*.
\end{equation}
The first shows that $q^*$ is injective and that $|G|^{-1}\mu^*$ is a
left inverse; the coefficient hypothesis is what makes the factor
$|G|^{-1}$ available. For surjectivity onto the invariant subalgebra,
let $x\in H^*(|K|;k)^{G}$. Since $g^*(x)=x$ for every $g\in G$, the second
identity in \eqref{eq:transferids} gives
$(q^*\circ \mu^*)(x)=\sum_{g\in G}g^*(x)=|G|\,x$, whence
$x=q^*\bigl(|G|^{-1}\mu^*(x)\bigr)$. This shows $q^*$ is an isomorphism. 
Being induced by a continuous map, $q^*$ respects cup products, so
\eqref{eq:transfer} is an isomorphism of graded algebras. The
homeomorphism $|K|/G\cong|K/G|$ identifies it with the stated map.
\end{proof}

\begin{remark}
The transfer in this form is due to Floyd and Conner; see the attribution
in \cite[Ch.~III, Sec.~2, p.~121]{Bredon1972}. We use only the case of a single
finite group with coefficients in a field in which $|G|$ is invertible.
\end{remark}

\subsubsection{Application to symmetric products of simplicial complexes}
\label{subsec:symprod}

\begin{proposition}[Orbit isomorphism for a locally finite simplicial complex]
\label{prop:orbitfin}
Let $K$ be a locally finite simplicial complex and $m\ge 1$. The map $q_K^*$ induced by
the orbit map $q_K\colon|K|^{m}\to\SP^{m}(|K|)$ is injective with image
the invariant subalgebra, so that it corestricts to an isomorphism of
graded rational algebras
\begin{equation}\label{eq:orbitfin}
\bar q_K^*\colon H^*\!\bigl(\SP^{m}(|K|);\Q\bigr)\iso
H^*\!\bigl(|K|^{m};\Q\bigr)^{\Sig_m}.
\end{equation}
\end{proposition}

\begin{proof}
Since $K$ is locally finite, $|K|^{m}$ has regular CW structure
whose open cells are the products
$\mathring\sigma_1\times\dots\times\mathring\sigma_m$, with
$\sigma_i\in K$. Note that $\Sig_m$ acts cellularly on $|K|^{m}$ by permuting product cells among
themselves. For $m\ge 2$, however, the action need not satisfy Bredon's condition
\textup{(A)}: if some factors coincide, an element of $\Sig_m$ may
stabilise a product cell setwise without fixing it pointwise. For
instance, a transposition stabilises
$\mathring\sigma\times\mathring\sigma$ setwise while sending
$(x,y)$ to $(y,x)$. We therefore pass below to a barycentric
subdivision of this cellular structure.

Let $\mathcal F=\mathcal F(|K|^m)$ be the face poset of this regular CW
structure, i.e.\ the poset of closures of nonempty cells ordered by
inclusion, and let $\Delta(\mathcal F)$ be its order complex. The face
poset is functorial for cellular maps of regular CW complexes
\cite[Def.~10.11]{Kozlov2008}, so the cellular $\Sig_m$-action induces an
action on $\mathcal F$ by poset automorphisms, and hence a simplicial
$\Sig_m$-action on $\Delta(\mathcal F)$. Moreover the barycentric
subdivision of the regular CW complex is exactly this order complex,
\begin{equation}\label{eq:bdorder}
\operatorname{Bd}(|K|^m)=\Delta(\mathcal F),
\qquad
|\Delta(\mathcal F)|\cong|K|^m
\end{equation}
\cite[Eq.~(10.4) and~(10.5)]{Kozlov2008}. 
This homeomorphism can be
chosen $\Sig_m$-equivariant. Indeed, for each product cell
$C=\sigma_1\times\cdots\times\sigma_m$, choose the product of the
barycentres of its factors as its distinguished interior point.
These choices are preserved by the coordinate-permutation action,
and hence the resulting barycentric subdivision is
$\Sig_m$-equivariant.
Write $T=\Delta(\mathcal F)$ for this equivariant simplicial
complex.
Then $|T|\cong|K|^m$ equivariantly, and hence
$$
|T|/\Sig_m\cong |K|^m/\Sig_m=\SP^m(|K|).
$$

It remains to secure Bredon's regularity. We first observe that the
simplicial $\Sig_m$-action on $T$ satisfies condition \textup{(A)}.
Indeed, a vertex of $T$ is the closure $C$ of a nonempty cell of $|K|^m$. If $C$ and
$\sigma C$, with $\sigma\in\Sig_m$, belong to a common simplex of $T$,
then they are comparable in the face poset $\mathcal F$. Since the
$\Sig_m$-action preserves the dimension of cells, this forces
$C=\sigma C$.
Hence condition \textup{(A)} holds on $T$.

Let $T'=\mathrm{sd}\,T$. Then by above Fact-1, the induced action on $T'$ is
regular. The canonical homeomorphism $|T'|\cong|T|\cong|K|^m$ is
$\Sig_m$-equivariant, and therefore
$$
|T'|/\Sig_m\cong|K|^m/\Sig_m=\SP^m(|K|).
$$
Lemma~\ref{lem:transfer}, applied to the regular $\Sig_m$-complex $T'$
with $k=\Q$, now yields \eqref{eq:orbitfin}.
\end{proof}

\begin{remark}\label{rem:coeff}
In Proposition \ref{prop:orbitfin}, only invertibility of $m!$ in the coefficient field is used. For fixed
$m$, one may replace $\Q$ by any field $k$ with
$\operatorname{char}(k)\nmid m!$. Since the applications to $\dsecat$
involve $m=r!$ as $r$ varies, rational coefficients provide a natural
uniform choice.
\end{remark}

\subsection{Extension to spaces of locally finite simplicial homotopy type}
\label{sec:cwtype}

Recall that a space is of \emph{locally finite simplicial homotopy type} if it is homotopy equivalent to the geometric realisation of a locally finite simplicial complex.

\begin{theorem}[Orbit isomorphism for locally finite simplicial homotopy type]
\label{thm:orbitcw}
Suppose $X$ is of locally finite simplicial homotopy type and $m\ge 1$. Then the
homomorphism of graded rational algebras induced by the orbit map,
\[
q_X^*\colon H^*(\SP^m(X);\Q)\longrightarrow H^*(X^m;\Q),
\]
is injective, and its image is exactly the invariant subalgebra:
\begin{equation}\label{eq:orbitcw}
\im q_X^*=H^*(X^m;\Q)^{\Sig_m}.
\end{equation}
Consequently, $q_X^*$ corestricts to an isomorphism of graded rational
algebras
\begin{equation}\label{eq:orbitcores}
\bar q_X^*\colon H^*(\SP^m(X);\Q)\iso H^*(X^m;\Q)^{\Sig_m},
\end{equation}
which differs from $q_X^*$ only in the declared codomain: for every
$a\in H^*(\SP^m(X);\Q)$ one has $\bar q_X^*(a)=q_X^*(a)$.
\end{theorem}

\begin{proof}
Since $X$ is of locally finite simplicial homotopy type, there is a locally finite simplicial complex
$K$ and a homotopy equivalence $h\colon|K|\to X$; fix such a pair.
The square
\[
\begin{array}{ccc}
|K|^m & \xrightarrow{\ q_K\ } & \SP^m(|K|)\\[2pt]
{\scriptstyle h^m}\big\downarrow & & \big\downarrow{\scriptstyle\SP^m(h)}\\[4pt]
X^m & \xrightarrow{\ q_X\ } & \SP^m(X)
\end{array}
\]
commutes strictly by \eqref{eq:strict}. Both $h^m$ and $\SP^m(h)$ are homotopy equivalences, so they induce isomorphisms in cohomology. Moreover $h^m$ is
$\Sig_m$-equivariant. Therefore $(h^m)^*$ is an isomorphism of $\Sig_m$-modules and hence
restricts to an isomorphism
$$H^*(X^m;\Q)^{\Sig_m}\iso H^*(|K|^m;\Q)^{\Sig_m}.$$

We chase the induced cohomology square
\[
\begin{array}{ccc}
H^*(\SP^m(X);\Q) & \xrightarrow{\ q_X^*\ } & H^*(X^m;\Q)\\[2pt]
{\scriptstyle\SP^m(h)^*}\big\downarrow & & \big\downarrow{\scriptstyle(h^m)^*}\\[4pt]
H^*(\SP^m(|K|);\Q) & \xrightarrow{\ q_K^*\ } & H^*(|K|^m;\Q).
\end{array}
\]
First, $\im q_X^*\subseteq H^*(X^m;\Q)^{\Sig_m}$ because $q_X\circ \sigma=q_X$
for all $\sigma$. For injectivity, if $q_X^*(a)=0$ then
$q_K^*(\SP^m(h)^*a)=(h^m)^*q_X^*(a)=0$; since $q_K^*$ is injective
(Proposition~\ref{prop:orbitfin}) and $\SP^m(h)^*$ is an isomorphism,
$a=0$. For the reverse inclusion in \eqref{eq:orbitcw}, take
$v\in H^*(X^m;\Q)^{\Sig_m}$. Since $h^m$ is equivariant, so is $(h^m)^*$,
whence $(h^m)^*v$ is again invariant; by
Proposition~\ref{prop:orbitfin} there is $b$ with
$q_K^*(b)=(h^m)^*v$. Choose $a$ with $\SP^m(h)^*a=b$ (possible, as
$\SP^m(h)^*$ is an isomorphism). Then
$(h^m)^*q_X^*(a)=q_K^*\SP^m(h)^*(a)=q_K^*(b)=(h^m)^*v$; since $(h^m)^*$ is
injective on all of $H^*(X^m;\Q)$, we get $q_X^*(a)=v$. Thus
\eqref{eq:orbitcw} holds. Finally, $q_X^*$ is an algebra map because it is
induced by a continuous map, and the corestriction
\eqref{eq:orbitcores} is then an isomorphism of graded algebras onto
$H^*(X^m;\Q)^{\Sig_m}$, the latter being a subalgebra since the
$\Sig_m$-action is by algebra automorphisms.
\end{proof}

\begin{remark}\label{rem:nomap}
Theorem~\ref{thm:orbitcw} is proved separately for each space $X$. Once
\eqref{eq:orbitcw} is known, the splitting below is defined directly on
$X$ and is natural for every continuous map. Thus there is no need to
replace a fibration $p\colon E\to B$ by a simplicial representative, and 
the kernel identity is obtained with no appeal to the homotopy invariance
of $\dsecat$; that invariance enters only later, in the application of
Section~\ref{sec:dsecat}.
\end{remark}

\subsection{The natural rational splitting}
\label{sec:splitting}

\begin{proposition}[Natural symmetrisation splitting]
\label{prop:split}
Let $X$ be of locally finite simplicial homotopy type and $m\ge 1$.
Then there exists a graded $\Q$-linear map
\[
s^X_m\colon H^*(X;\Q)\longrightarrow H^*(\SP^m(X);\Q)
\]
such that
\begin{equation}\label{eq:section}
(\sd^X_m)^*\circ s^X_m=\id_{H^*(X;\Q)}.
\end{equation}
Moreover, it is natural: for every map
$f\colon X\to Y$ between spaces of locally finite simplicial
homotopy type,
\begin{equation}\label{eq:natsplit}
\SP^m(f)^*\circ s^Y_m=s^X_m\circ f^*.
\end{equation}
\end{proposition}

\begin{proof}
For $u\in H^*(X;\Q)$, define
\begin{equation}\label{eq:avg}
A^X_m(u)=\frac1m\sum_{j=1}^m(\pr^X_j)^*(u)\ \in\ H^*(X^m;\Q).
\end{equation}
We first check that $A^X_m(u)$ is $\Sig_m$-invariant. For
$\sigma\in\Sig_m$ we have $\pr^X_j\circ\sigma=\pr^X_{\sigma^{-1}(j)}$,
hence
\[
\sigma^*A^X_m(u)=\frac1m\sum_{j=1}^m(\pr^X_{\sigma^{-1}(j)})^*(u)
=A^X_m(u),
\]
since $\sigma^{-1}$ permutes the index set. Thus
$A^X_m(u)\in H^*(X^m;\Q)^{\Sig_m}$. No Koszul sign arises: each
summand $(\pr^X_j)^*(u)$ is a single class, so we permute summands,
not factors in a cup product.

By Theorem~\ref{thm:orbitcw}, the corestricted orbit pullback
\[
\bar q_X^*\colon H^*(\SP^m(X);\Q)
\longrightarrow H^*(X^m;\Q)^{\Sig_m}
\]
is an isomorphism. Since $A^X_m(u)$ is invariant, we may therefore
define
\begin{equation}\label{eq:defsplit}
s^X_m:=(\bar q_X^*)^{-1}\circ A^X_m\colon
H^*(X;\Q)\longrightarrow H^*(\SP^m(X);\Q).
\end{equation}
Linearity and the grading are inherited from $A^X_m$ and
$(\bar q_X^*)^{-1}$, both of which are graded and $\Q$-linear.

The full orbit pullback
$q_X^*\colon H^*(\SP^m(X);\Q)\to H^*(X^m;\Q)$
agrees with $\bar q_X^*$ on classes in its domain, with the invariant
subalgebra viewed as a subalgebra of $H^*(X^m;\Q)$. Consequently,
the definition of $s^X_m$ gives
\begin{equation}\label{eq:defsplitchar}
q_X^*\,s^X_m(u)=A^X_m(u).
\end{equation}
Moreover, $q_X^*$ is injective by Theorem~\ref{thm:orbitcw}, so
this identity uniquely determines $s^X_m(u)$. In particular, the
construction does not depend on the choice of simplicial model
$(K,h)$ used in Theorem~\ref{thm:orbitcw}.

Because $\sd^X_m=q_X\diag_X$ and
$\pr^X_j\diag_X=\id_X$, using \eqref{eq:defsplitchar},
\[
(\sd^X_m)^*s^X_m(u)=\diag_X^*q_X^*s^X_m(u)
=\frac1m\sum_{j=1}^m(\pr^X_j\diag_X)^*(u)
=\frac1m\sum_{j=1}^m u=u,
\]
which proves \eqref{eq:section}.

For naturality, let $f\colon X\to Y$ and $u\in H^*(Y;\Q)$.
The naturality of the orbit maps gives
$(f^m)^*q_Y^*=q_X^*\SP^m(f)^*$, and
$\pr^Y_j\circ f^m=f\circ\pr^X_j$ gives
$(f^m)^*A^Y_m=A^X_m\circ f^*$. Using
\eqref{eq:defsplitchar}, we obtain
\[
\begin{aligned}
q_X^*\,\SP^m(f)^*s^Y_m(u)
&=(f^m)^*q_Y^*s^Y_m(u)\\
&=(f^m)^*A^Y_m(u)\\
&=A^X_m(f^*u)\\
&=q_X^*s^X_m(f^*u).
\end{aligned}
\]
Here the full pullbacks are used, with codomain the whole of
$H^*(X^m;\Q)$, so they can be composed with $(f^m)^*$.
Since $q_X^*$ is injective, this proves \eqref{eq:natsplit}.
\end{proof}

\begin{remark}\label{rem:whym}
The normalisation $1/m$ in \eqref{eq:avg} is what makes $s^X_m$ a
section rather than $m$ times one: restricting along the diagonal sends
each of the $m$ summands to the same class $u$, so the unnormalised sum
    would give $(\sd^X_m)^*s^X_m=m\cdot\id_{H^*(X;\Q)}$. At this normalization step, rational coefficients are used to invert $m$. The same integer is inverted, from a different angle, in \cite[Prop.~4.3]{D-J}: there one uses $\xi^m_\infty\sd_m=m\,\xi_\infty$ in $\SP^\infty(X)$ and the fact that multiplication by $m$ is a rational homology isomorphism. Note that
$m$ and $|\Sig_m|=m!$ play distinct roles: the former normalises the
average over the $m$ coordinate projections, the latter the average over
the group in Lemma~\ref{lem:transfer}. Both are invertible in $\Q$.
\end{remark}

\begin{remark}\label{rem:mult}
The pullback $(\sd^X_m)^*\colon H^*(\SP^m(X);\Q)\to H^*(X;\Q)$ is a
homomorphism of graded algebras, being induced by the continuous map
$\sd^X_m$. This is what is used in the cup-length argument below
(Corollary~\ref{cor:lift} and Theorem~\ref{thm:dsecat}).
\end{remark}

\begin{remark}\label{rem:notmult}
In general, $s^X_m$ is not an algebra homomorphism. This is irrelevant to
the cup-length argument: what is used there is that $(\sd^X_m)^*$ is
multiplicative (Remark~\ref{rem:mult}) and that $(\sd^X_m)^*s^X_m=\id_{H^*(X;\Q)}$.
\end{remark}

\subsection{The kernel identity}
\label{sec:kernel}

\begin{theorem}[Kernel identity]\label{thm:kernel}
Let $m\ge 1$ and let $p\colon E\to B$ be a continuous map between spaces
of locally finite simplicial homotopy type. Then
\begin{equation}\label{eq:kernel}
(\sd^B_m)^*\bigl(\ker \SP^m(p)^*\bigr)=\ker p^*.
\end{equation}
Moreover $(\sd^B_m)^*\colon\ker \SP^m(p)^*\to\ker p^*$ is a split
epimorphism, with section $s^B_m|_{\ker p^*}$.
\end{theorem}

\begin{proof}
Naturality of the symmetric diagonal gives the strictly commutative
square $\sd^B_m\circ p=\SP^m(p)\circ\sd^E_m$, whence in cohomology
\begin{equation}\label{eq:squarecoh}
p^*(\sd^B_m)^*=(\sd^E_m)^*(\SP^m(p))^*.
\end{equation}
If $\alpha\in\ker \SP^m(p)^*$, then
$p^*(\sd^B_m)^*(\alpha)=(\sd^E_m)^*(\SP^m(p))^*(\alpha)=0$, so
$(\sd^B_m)^*(\alpha)\in\ker p^*$; this is the inclusion $\subseteq$.
Conversely, let $u\in\ker p^*$ and put $\alpha=s^B_m(u)$. By naturality
of the splitting, Proposition~\ref{prop:split},
\[
(\SP^m(p))^*(\alpha)=\SP^m(p)^*s^B_m(u)=s^E_m(p^*u)=s^E_m(0)=0,
\]
so $\alpha\in\ker \SP^m(p)^*$; and $(\sd^B_m)^*(\alpha)=u$
by~\eqref{eq:section}. This is the inclusion $\supseteq$ and exhibits
$s^B_m|_{\ker p^*}$ as a section.
\end{proof}

\begin{remark}\label{rem:nofib}
Theorem~\ref{thm:kernel} identifies $\ker p^*$ exactly as the image under
$(\sd^B_m)^*$ of the kernel upstairs; it is a functorial statement, using
no fibration hypothesis.
\end{remark}

\begin{remark}\label{rem:naturality}
The essential point is the naturality of the splitting, not merely
its existence. The map $(\sd^B_m)^*$ is a split epimorphism, so a section
$s$ exists on abstract grounds; but for the argument above one needs
$\SP^m(p)^*s^B_m(u)=s^E_m(p^*u)$, which forces $s^B_m(u)\in\ker\SP^m(p)^*$
whenever $u\in\ker p^*$. An arbitrary section, chosen independently for
each space, need not make this square commute, and then there is no
reason for the lift of a kernel class to land in the kernel upstairs. The
explicit symmetrisation of Proposition~\ref{prop:split} is precisely what
supplies a section that is natural for every continuous map, and hence
makes \eqref{eq:kernel} hold.
\end{remark}

\subsection{Cup length with rational coefficients}
\label{sec:cuplength}
In this subsection, for a continuous map $p\colon E\to B$, we compare the
cup lengths of the two graded ideals $\ker \SP^m(p)^*$ and $\ker p^*$.

\begin{definition}\label{def:cuplength}
For a graded ideal $I$ in a graded $\Q$-algebra, its cup
length is
\[
\cl_\Q(I):=\sup\bigl\{\,r\ge 0:\exists\,u_1,\dots,u_r\in I
\text{ homogeneous},\ |u_i|>0,\ u_1\cdots u_r\ne0\,\bigr\},
\]
allowed to be infinite.
\end{definition}

\begin{corollary}[Term-by-term lifting of nonzero products]
\label{cor:lift}
Under the hypotheses of Theorem~\ref{thm:kernel}, we have
$$\cl_\Q\bigl(\ker \SP^m(p)^*\bigr)\ge\cl_\Q(\ker p^*).$$ More precisely,
every nonzero product in $\ker p^*$ lifts term by term to a nonzero
product in $\ker \SP^m(p)^*$.
\end{corollary}

\begin{proof}
Let $u_1,\dots,u_r\in\ker p^*$ be homogeneous classes of positive degree
with $u_1\cdots u_r\ne0$, and set $\alpha_i=s^B_m(u_i)$, so
$\alpha_i\in\ker \SP^m(p)^*$ by Theorem~\ref{thm:kernel}. Since
$(\sd^B_m)^*$ is an algebra homomorphism and
$(\sd^B_m)^*(\alpha_i)=u_i$,
\[
(\sd^B_m)^*(\alpha_1\cdots\alpha_r)=\prod_{i=1}^r(\sd^B_m)^*(\alpha_i)
=u_1\cdots u_r\ne0,
\]
so $\alpha_1\cdots\alpha_r\ne0$. Note that $s^B_m$ itself is not assumed
multiplicative; see Remark~\ref{rem:notmult}.
\end{proof}

\subsection{Application to distributional sectional category}\label{sec:dsecat}

Throughout this subsection, $B$ and $E$ are of \emph{locally finite CW type}.
This strengthens the standing hypothesis of the previous sections: a
locally finite CW complex is homotopy equivalent to a locally finite simplicial complex
\cite[Theorem~1]{Milnor1959}, so a space of locally finite CW type is of
locally finite simplicial homotopy type, and
Proposition~\ref{prop:split} and Theorem~\ref{thm:kernel} apply to it
unchanged. The strengthening is what the distributional theory requires:
a locally finite CW complex is metrizable, indeed a metric
ANR, so the distributional constructions---which are built from the
Lévy--Prokhorov metric on spaces of finitely supported measures---are
available on a model, and are transported to $B$ and $E$ by the homotopy
invariance of $\dsecat$.

We recall the one input needed from the theory of distributional
sectional category. 
Theorem~\ref{cohomological lower bound on dsecat}
is formulated in Alexander--Spanier cohomology to
exploit the tautness of arbitrary subspaces of
metrizable spaces
\cite[Theorem~1(3)]{SpanierTautness1978}.
In its proof, this property allows cohomology
classes vanishing on the subsets $A_i$ to vanish
on suitable open neighbourhoods.

For spaces of locally finite CW homotopy type, Alexander--Spanier
and singular cohomology agree naturally as graded algebras \cite[Chapter~6]{Sp}.
Moreover, finite symmetric products preserve homotopy equivalences,
and the symmetric products of locally finite simplicial complexes
admit locally finite CW structures. Thus the comparison also
applies to the finite symmetric products occurring below.
Consequently, Theorem~\ref{cohomological lower bound on dsecat}
may be applied using singular cohomology throughout this subsection.

\begin{theorem}[Rational kernel cup-length bound]\label{thm:dsecat}
Let $p\colon E\to B$ be a fibration between path-connected spaces of
locally finite CW type, with $E$ metric. Then
\begin{equation}\label{eq:dsecat}
\dsecat(p)\ge\cl_\Q(\ker p^*).
\end{equation}
\end{theorem}

\begin{proof}
Choose a homotopy equivalence
$v\colon\widetilde B\to B$, where $\widetilde B$ is a locally finite CW
complex, and let
$q\colon\widetilde E\to\widetilde B$
be the pullback of $p$ along $v$, with induced map
$w\colon\widetilde E\to E$. Thus
$p\circ w=v\circ q$ strictly. Since $p$ is a fibration and $v$ is a homotopy equivalence, $w$ is a homotopy equivalence. Moreover $\widetilde E$ is metric, being a subspace of
$\widetilde B\times E$, and it is of locally finite CW type since $w\colon\widetilde E\to E$ is a homotopy equivalence. Hence, by homotopy invariance,
\[
\dsecat(p)=\dsecat(q).
\]
In cohomology, $w^*\circ p^*=q^*\circ v^*.$
Since $v^*$ and $w^*$ are isomorphisms of graded algebras, $v^*$
restricts to an isomorphism
$\ker p^*\cong\ker q^*,$
and therefore
$$
\cl_\Q(\ker p^*)=\cl_\Q(\ker q^*).
$$
It therefore suffices to prove $\dsecat(q)\ge\cl_\Q(\ker q^*)$, and we
argue on the model $\widetilde B$, which is of locally finite CW type.

Let $r\ge1$ and suppose there are homogeneous positive-degree classes
$u_1,\dots,u_r\in\ker q^*$ with $u_1\cdots u_r\ne0$. Apply
Proposition~\ref{prop:split} with $m=r!$ and set
$\beta_i=s^{\widetilde B}_{r!}(u_i)$. By Theorem~\ref{thm:kernel},
$\beta_i\in\ker\SP^{r!}(q)^*,$ and
\[
\prod_{i=1}^r(\sd^{\widetilde B}_{r!})^*(\beta_i)=\prod_{i=1}^r u_i\ne0,
\]
using \eqref{eq:section} and the multiplicativity of
$(\sd^{\widetilde B}_{r!})^*$ (Remark~\ref{rem:mult}).
Theorem~\ref{cohomological lower bound on dsecat} gives $\dsecat(q)\ge r$. Taking the
supremum over all such $r$, and using
$\dsecat(p)=\dsecat(q)$ together with
$\cl_\Q(\ker p^*)=\cl_\Q(\ker q^*)$, proves
\eqref{eq:dsecat}; if the cup length is infinite, the same argument for
every $r$ gives $\dsecat(p)=\infty$.
\end{proof}

This completes the proof of Theorem~\ref{thm:main}: the splitting and
kernel identity are Proposition~\ref{prop:split} and
Theorem~\ref{thm:kernel}, and the bound is Theorem~\ref{thm:dsecat}.

We now compute the distributional sectional category of the Hopf fibration.
\begin{example}\label{ex: Hopf fibration}
\normalfont{
Consider the Hopf fibration
$S^1\longrightarrow S^{2m+1}\xrightarrow{p}\mathbb{C}\mathrm{P}^m.$
The rational cohomology ring of the base is
$$
H^*(\mathbb{C}\mathrm{P}^m;\mathbb{Q})
\cong
\mathbb{Q}[u]/(u^{m+1}),
\qquad |u|=2.
$$
Since $H^2(S^{2m+1};\Q)=0$, we have $u\in\ker p^*$, and
$u^m\neq0$. Hence Theorem~\ref{thm:dsecat} gives
$\dsecat(p)\ge m$. On the other hand, Corollary~\ref{cor-dsecat-dcat} and
\cite[Proposition~6.6]{D-J} give
$\dsecat(p)\le\dcat(\mathbb{C}\mathrm{P}^m)=m$. Therefore
$\dsecat(p)=m$.}
\end{example}

\begin{example}
\normalfont{
Recall that the lens space $L^{2n+1}_m=S^{2n+1}/\Z_m$ is obtained by letting
$\Z_m\subset S^1$ act on $S^{2n+1}\subseteq\C^{n+1}$ by scalar
multiplication. The standard scalar $S^1$-action on $S^{2n+1}$ therefore
induces a free action of the quotient group $S^1/\Z_m$ on
$L^{2n+1}_m$, whose orbit space is $\mathbb{C}\mathrm{P}^n$; see
\cite[Section~2]{J}. Since $S^1/\Z_m\cong S^1$, we obtain a principal
circle bundle
\[
S^1\longrightarrow L^{2n+1}_m\xrightarrow{p}\mathbb{C}\mathrm{P}^n.
\]

With rational coefficients,
$H^*(L^{2n+1}_m;\Q)\cong H^*(S^{2n+1};\Q)$, while
$H^*(\mathbb{C}\mathrm{P}^n;\Q)\cong\Q[u]/(u^{n+1})$ with $|u|=2$. Since
$H^2(L^{2n+1}_m;\Q)=0$, we have $u\in\ker p^*$, and since
$u^n\neq0$, Theorem~\ref{thm:dsecat} gives $\dsecat(p)\ge n$.
On the other hand, Corollary~\ref{cor-dsecat-dcat} and
\cite[Proposition~6.6]{D-J} give
$\dsecat(p)\le\dcat(\mathbb{C}\mathrm{P}^n)=n$. Hence
$\dsecat(p)=n$.
}
\end{example}

\section{Multiplicative product inequality for distributional sectional category}\label{sec: multiplicative prod ineq}

In this section, we establish a multiplicative product inequality for the distributional sectional category. As a consequence, we obtain such inequalities for $\dcat$ and $\dTC$.

External products of finitely supported probability measures appear in \cite[Proposition~3.5]{K-W}, with the quotient topology, and measure multiplication maps have been used in \cite[Proposition~3.9]{Jau1} and
\cite[Theorem~8.1]{Dau-J}. Here we prove continuity of the tensor product in the L\'evy--Prokhorov topology and use it to establish a general multiplicative inequality for distributional sectional category.

\begin{lemma}[Support of a product measure]\label{lem:product-support}
Let $X$ and $Y$ be metric spaces, and let
$\mu=\sum_{i\in I}\alpha_i\delta_{x_i}\in\mathcal{B}_m(X)$ and
$\nu=\sum_{j\in J}\beta_j\delta_{y_j}\in\mathcal{B}_n(Y)$ be probability measures
with $|I|\leq m$, $|J|\leq n$, the points $x_i$ (resp.\ $y_j$) pairwise distinct, and all weights strictly positive. Then
\[
\mu\otimes\nu=\sum_{(i,j)\in I\times J}\alpha_i\beta_j\,\delta_{(x_i,y_j)}
\]
is a probability measure with
\[
\operatorname{supp}(\mu\otimes\nu)=\operatorname{supp}(\mu)\times\operatorname{supp}(\nu),
\qquad
|\operatorname{supp}(\mu\otimes\nu)|\leq mn.
\]
In particular $\mu\otimes\nu\in\mathcal{B}_{mn}(X\times Y)$.
\end{lemma}

\begin{proof}
Each weight $\alpha_i\beta_j$ is strictly positive, so no atom cancels, and the support is exactly $\{(x_i,y_j):i\in I,\,j\in J\}=\operatorname{supp}(\mu)\times\operatorname{supp}(\nu)$, of cardinality at most $mn$. The total mass is $\big(\sum_i\alpha_i\big)\big(\sum_j\beta_j\big)=1$, so $\mu\otimes\nu$ is a probability measure.
\end{proof}

\begin{lemma}[Continuity of the tensor product]\label{lem:tensor-continuity}
Let $X$ and $Y$ be separable metric spaces. Then the map
\[
\otimes:\mathcal{B}_m(X)\times \mathcal{B}_n(Y)\longrightarrow
\mathcal{B}_{mn}(X\times Y),\qquad
(\mu,\nu)\longmapsto \mu\otimes\nu,
\]
is well defined and continuous, where $\mathcal{B}_m(X)$, $\mathcal{B}_n(Y)$ and $\mathcal{B}_{mn}(X\times Y)$ carry the subspace topologies induced by the L\'evy--Prokhorov metric on the corresponding spaces of Borel probability measures.
\end{lemma}

\begin{proof}
The map is well defined by Lemma~\ref{lem:product-support}: for $\mu\in\mathcal{B}_m(X)$ and $\nu\in\mathcal{B}_n(Y)$, the product $\mu\otimes\nu$ has support of cardinality at most $mn$, hence lies in $\mathcal{B}_{mn}(X\times Y)$.

Since $X$ and $Y$ are separable metric spaces, the L\'evy--Prokhorov topology agrees with the topology of weak convergence. The spaces involved are metrizable, so it is enough to prove sequential continuity.

Let $\mu_k\to\mu$ in $\mathcal{B}_m(X)$ and $\nu_k\to\nu$ in $\mathcal{B}_n(Y)$. Write $\mu=\sum_{i=1}^r\alpha_i\delta_{x_i}$ and $\nu=\sum_{j=1}^s\beta_j\delta_{y_j}$, with $r\leq m$, $s\leq n$, the points $x_i$ and $y_j$ pairwise distinct, and all coefficients positive. By Lemma~\ref{lem:product-support}, $\mu_k\otimes\nu_k\in\mathcal{B}_{mn}(X\times Y)$ for every $k$ and $\mu\otimes\nu\in\mathcal{B}_{mn}(X\times Y)$; we prove that $\mu_k\otimes\nu_k$ converges weakly to $\mu\otimes\nu$.

Let $f\in C_b(X\times Y)$ and set $M=\|f\|_\infty$. Fix $\varepsilon>0$. By continuity of $f$ at the finitely many points $(x_i,y_j)$, we may choose $\rho>0$ such that the balls $U_i=B(x_i,\rho)$ are pairwise disjoint, the balls $V_j=B(y_j,\rho)$ are pairwise disjoint, and $|f(x,y)-f(x_i,y_j)|<\varepsilon$ whenever $x\in U_i$ and $y\in V_j$.

Put $U=\bigcup_i U_i$ and $V=\bigcup_j V_j$. We first claim that $\mu_k(U_i)\to\alpha_i$ for every $i$, and $\mu_k(X\setminus U)\to0$; similarly, $\nu_k(V_j)\to\beta_j$ for every $j$, and $\nu_k(Y\setminus V)\to0$. Indeed, since each $U_i$ is open and the points $x_i$ are distinct with the $U_i$ pairwise disjoint, $\mu(U_i)=\alpha_i$; the portmanteau theorem then gives $\liminf_k\mu_k(U_i)\geq\mu(U_i)=\alpha_i$. Since $U$ is open and $\mu(U)=1$, it also gives $\mu_k(U)\to1$. As the $U_i$ are pairwise disjoint, $\sum_i\mu_k(U_i)=\mu_k(U)\to1$. Hence every convergent subsequence of $(\mu_k(U_i))_i$ has limit $(c_i)_i$ with $c_i\geq\alpha_i$ and $\sum_i c_i=1=\sum_i\alpha_i$, forcing $c_i=\alpha_i$ for all $i$. Thus $\mu_k(U_i)\to\alpha_i$ for every $i$, and consequently $\mu_k(X\setminus U)=1-\mu_k(U)\to0$. The proof for $\nu_k$ is identical.

Let $R=(X\times Y)\setminus(U\times V)$. Since $R\subseteq ((X\setminus U)\times Y)\cup (X\times(Y\setminus V))$, we have
\[
(\mu_k\otimes\nu_k)(R)\leq \mu_k(X\setminus U)+\nu_k(Y\setminus V)\longrightarrow 0 .
\]
The set $U\times V$ is the disjoint union of the rectangles $U_i\times V_j$, and on each of them $f$ differs from the constant $f(x_i,y_j)$ by less than $\varepsilon$, while $(\mu_k\otimes\nu_k)(U_i\times V_j)=\mu_k(U_i)\nu_k(V_j)$. Therefore
\begin{multline*}
\left|
\int f\,d(\mu_k\otimes\nu_k)
-
\sum_{i,j} f(x_i,y_j)\mu_k(U_i)\nu_k(V_j)
\right|\\
\leq
\varepsilon\,(\mu_k\otimes\nu_k)(U\times V)+M\,(\mu_k\otimes\nu_k)(R)
\leq
\varepsilon + M\,(\mu_k\otimes\nu_k)(R).
\end{multline*}
Letting $k\to\infty$ and using $\mu_k(U_i)\nu_k(V_j)\to\alpha_i\beta_j$, we obtain
\[
\limsup_{k\to\infty}
\left|
\int f\,d(\mu_k\otimes\nu_k)
-
\sum_{i,j} f(x_i,y_j)\alpha_i\beta_j
\right|
\leq \varepsilon .
\]
Since $\varepsilon>0$ is arbitrary and $\sum_{i,j} f(x_i,y_j)\alpha_i\beta_j=\int f\,d(\mu\otimes\nu)$, it follows that $\int f\,d(\mu_k\otimes\nu_k)\to\int f\,d(\mu\otimes\nu)$. Thus $\mu_k\otimes\nu_k$ converges weakly to $\mu\otimes\nu$, and therefore converges in the L\'evy--Prokhorov topology.
\end{proof}

\begin{theorem}\label{prop:product-inequality}
Let $p\colon E\to B$ and $p'\colon E'\to B'$ be fibrations, where $E$ and $E'$ are separable metric spaces. Then
\[
\operatorname{dsecat}(p\times p')
\leq
(\operatorname{dsecat}(p)+1)(\operatorname{dsecat}(p')+1)-1 .
\]
\end{theorem}

\begin{proof}
If one of the two invariants is infinite, the assertion is trivial. Thus we may assume that $a=\operatorname{dsecat}(p)<\infty$ and $c=\operatorname{dsecat}(p')<\infty$. Put $m=a+1$ and $n=c+1$.

By definition, $\mathcal{B}_m(p)\colon E_m(p)\to B$ and $\mathcal{B}_n(p')\colon E_n(p')\to B'$ admit sections $s:B\to E_m(p)$ and $s'\colon B'\to E_n(p')$. Hence, for each $x\in B$, the measure $s(x)$ is supported in $p^{-1}(x)$ and has support of cardinality at most $m$; similarly, for each $y\in B'$, the measure $s'(y)$ is supported in $(p')^{-1}(y)$ and has support of cardinality at most $n$.

Define $\sigma\colon B\times B'\to \mathcal{B}_{mn}(E\times E')$ by $\sigma(x,y)=s(x)\otimes s'(y)$. By Lemma~\ref{lem:product-support}, the support of $\sigma(x,y)$ is
\[
\begin{split}
\operatorname{supp}(\sigma(x,y))
&=
\operatorname{supp}(s(x))\times \operatorname{supp}(s'(y))\\
&\subset
p^{-1}(x)\times (p')^{-1}(y)
=
(p\times p')^{-1}(x,y),
\end{split}
\]
of cardinality at most $mn$. Therefore $\sigma(x,y)$ belongs to $E_{mn}(p\times p')$, and so $\sigma$ defines a map $\sigma:B\times B'\to E_{mn}(p\times p')$.

It remains to check continuity. The map $\sigma$ is the composite
\[
B\times B'
\xrightarrow{s\times s'}
E_m(p)\times E_n(p')
\hookrightarrow
\mathcal{B}_m(E)\times \mathcal{B}_n(E')
\xrightarrow{\otimes}
\mathcal{B}_{mn}(E\times E').
\]
The first map is continuous because $s$ and $s'$ are continuous, the middle map is continuous because $E_m(p)$ and $E_n(p')$ carry the subspace topologies, and the last map is continuous by Lemma~\ref{lem:tensor-continuity}. Since the image lies in the subspace $E_{mn}(p\times p')$, $\sigma$ is continuous as a map into $E_{mn}(p\times p')$.

Finally, $\sigma(x,y)$ is supported in the fibre $(p\times p')^{-1}(x,y)$, hence $\mathcal{B}_{mn}(p\times p')(\sigma(x,y))=(x,y)$. Thus $\sigma$ is a section of $\mathcal{B}_{mn}(p\times p')$. Consequently,
\[
\operatorname{dsecat}(p\times p')\leq mn-1=(a+1)(c+1)-1,
\]
which is the desired inequality.
\end{proof}

The following corollary specializes our general product
inequality to distributional LS-category and topological
complexity. The inequality for distributional LS-category
was previously established by Dranishnikov and Jauhari
\cite{D-J}.

\begin{corollary}
If $X$ and $Y$ are path-connected separable metric spaces, then we have 
\begin{enumerate}
    \item $\dcat(X\times Y)+1\leq (\dcat(X)+1)(\dcat(Y)+1)$.   
    
    \item $\dTC(X\times Y)+1\leq (\dTC(X)+1)(\dTC(Y)+1)$.
\end{enumerate}
\end{corollary}

\begin{proof}
Since $X$ and $Y$ are separable metric spaces, the path spaces
$PX$, $PY$, $P_0X$, and $P_0Y$, endowed with the compact--open
topology, are separable metric spaces.

For the distributional LS-category, let
$p^X\colon P_0X\to X$ and $p^Y\colon P_0Y\to Y$ be the based path
fibrations. Under the canonical homeomorphism
$P_0(X\times Y)\cong P_0X\times P_0Y$, the based path fibration of
$X\times Y$ identifies with $p^X\times p^Y$. Hence
Theorem~\ref{prop:product-inequality} gives
\[
\dcat(X\times Y)+1
\leq(\dcat(X)+1)(\dcat(Y)+1).
\]

Similarly, under the canonical homeomorphisms
$P(X\times Y)\cong PX\times PY$ and
$(X\times Y)^2\cong X^2\times Y^2$, the free path fibration of
$X\times Y$ identifies with $\pi_X\times\pi_Y$. Therefore
\[
\dTC(X\times Y)+1
\leq(\dTC(X)+1)(\dTC(Y)+1).\qedhere
\]
\end{proof}

\begin{remark}
In contrast to the classical subadditivity
$\operatorname{secat}(p\times p')
\leq
\operatorname{secat}(p)+\operatorname{secat}(p')
$
of Schwarz (when $B\times B'$ is paracompact), the above argument naturally gives a multiplicative estimate
in the distributional setting: the tensor product of fibrewise measures
multiplies the sizes of their supports rather than adding them. Thus this
proof does not recover an additive estimate.

On the other hand, additive product inequalities fail for $\dcat$ and
$\dTC_m$: Jauhari and Oprea \cite[Example~7.20]{JO-dcat1} showed that
$\dcat(\Z_2\times\Z_2)=3>2=2\,\dcat(\Z_2)$, with analogous failures for
$\dTC_m$. Since $\dcat$ and $\dTC$ are special cases of $\dsecat$, no
additive product inequality of the classical form can hold for
$\dsecat$ in general.
\end{remark}

\section{Composition inequality for distributional sectional category}\label{sec: composition inequality}

All spaces below are assumed to be separable metric spaces, unless
otherwise stated. We use throughout the notation and conventions for
spaces of finitely supported probability measures introduced in
Section~\ref{sec: prelim}. In particular, on separable metric spaces the
L\'evy--Prokhorov topology agrees with the topology of weak convergence.

We shall also use that $\mathcal{B}_k(Z)$ is separable whenever $Z$ is
separable. Indeed, if $D\subseteq Z$ is a countable dense subset, then
the set of measures of the form
\[
 \sum_{i=1}^r q_i\delta_{d_i},
 \qquad
 r\leq k,\quad d_i\in D,\quad q_i\in\Q_{\geq0},
 \quad \sum_iq_i=1,
\]
is countable and dense in $\mathcal{B}_k(Z)$.

\subsection{The barycentre map}

\begin{lemma}[Continuity of the barycentre map]
Let $Z$ be a separable metric space, and let $m,n\geq 1$. Equip $\mathcal{B}_n(Z)$ with the Lévy--Prokhorov topology, and form $\mathcal{B}_m(\mathcal{B}_n(Z))$ with respect to this topology. Then the barycentre map
\[
 \bary\colon \mathcal{B}_m(\mathcal{B}_n(Z))\longrightarrow \mathcal{B}_{mn}(Z), \ 
 \text{defined by}\ 
 \bary\left(\sum_{j=1}^r a_j\delta_{\nu_j}\right)
 :=\sum_{j=1}^r a_j\nu_j,
\]
is a well-defined continuous map.
\end{lemma}

\begin{proof}
First, let us check that it is well-defined. Consider
\[
 \Theta=\sum_{j=1}^r a_j\delta_{\nu_j}\in \mathcal{B}_m(\mathcal{B}_n(Z)),
 \qquad r\leq m,
\]
where each $\nu_j\in \mathcal{B}_n(Z)$. We want to show that
$ \left|\supp\bigl(\bary(\Theta)\bigr)\right|\leq mn.$ Indeed, each $\nu_j$ is supported on at most $n$ points of $Z$, and therefore
\[
 \supp\bigl(\bary(\Theta)\bigr)
 \subseteq
 \bigcup_{j=1}^r \supp(\nu_j).
\]
This union has cardinality at most $mn$, so $\bary(\Theta)\in \mathcal{B}_{mn}(Z)$.

It remains to prove continuity. Since $Z$ is separable metric, $\mathcal{B}_n(Z)$ is also separable
metric by the observation above. Hence the Lévy--Prokhorov topology agrees with the weak topology both on $\mathcal{B}(Z)$ and on
$\mathcal{B}(\mathcal{B}_n(Z))$. Let
$\Theta_l\longrightarrow \Theta$
in $\mathcal{B}_m(\mathcal{B}_n(Z))$. We show that
$\bary(\Theta_l)\longrightarrow \bary(\Theta)$
weakly in $\mathcal{B}_{mn}(Z)$. In order to see this, let $h\in C_b(Z)$ be a bounded continuous map and define
\[
 \widehat h\colon \mathcal{B}_n(Z)\longrightarrow \mathbb R,
 \qquad
 \widehat h(\nu)=\int_Z h\,d\nu.
\]
The function $\widehat h$ is bounded and continuous because $\mathcal{B}_n(Z)$ carries the weak topology. Hence
\[
 \int_Z h\,d\bary(\Theta_l)
 =
 \int_{\mathcal{B}_n(Z)} \widehat h(\nu)\,d\Theta_l(\nu)
 \longrightarrow
 \int_{\mathcal{B}_n(Z)} \widehat h(\nu)\,d\Theta(\nu)
 =
 \int_Z h\,d\bary(\Theta).
\]
Thus $\bary(\Theta_l)\to \bary(\Theta)$ weakly, and hence in the Lévy--Prokhorov topology. This proves continuity.
\end{proof}

\begin{remark}
The point of the separability hypothesis is precisely to justify the identification between the Lévy--Prokhorov topology and the weak topology, both on $\mathcal{B}_n(Z)$ and then again on $\mathcal{B}_m(\mathcal{B}_n(Z))$. The notation $\mathcal{B}_m(\mathcal{B}_n(Z))$ means measures on the metric space $\mathcal{B}_n(Z)$, where $\mathcal{B}_n(Z)$ has already been endowed with its Lévy--Prokhorov topology.
\end{remark}

\subsection{A composition inequality}
We are now all set to prove our main result of the section.
\begin{theorem}[Composition inequality]\label{prop: composition inequality}
Let $W\xrightarrow{\theta}V\xrightarrow{\rho}E$
be fibrations between separable metric spaces. Then
\[
 \dsecat(\rho\circ\theta)+1
 \leq
 \bigl(\dsecat(\rho)+1\bigr)
 \bigl(\dsecat(\theta)+1\bigr).
\]
\end{theorem}

\begin{proof}
If either invariant on the right-hand side is infinite, there is nothing to prove. Put
\[
 m=\dsecat(\rho)+1,
 \qquad
 n=\dsecat(\theta)+1.
\]
By definition, there exist continuous sections
$\sigma\colon E\longrightarrow V_m(\rho)$
of
$\mathcal{B}_m(\rho)\colon V_m(\rho)\longrightarrow E
$
and
$\tau\colon V\longrightarrow W_n(\theta)$
of
$\mathcal{B}_n(\theta)\colon W_n(\theta)\longrightarrow V.$
Using the inclusions
$V_m(\rho)\subseteq \mathcal{B}_m(V)$ and $W_n(\theta)\subseteq \mathcal{B}_n(W),$
we regard $\sigma$ as a continuous map $E\to \mathcal{B}_m(V)$ and $\tau$ as a continuous map $V\to \mathcal{B}_n(W)$.

Define
\[
 s\colon E\longrightarrow \mathcal{B}_{mn}(W),
 \ \text{ by }\ 
 s(x):=\bary\bigl(\tau_*\sigma(x)\bigr),
\]
where
$\tau_*:\mathcal{B}_m(V)\longrightarrow \mathcal{B}_m(\mathcal{B}_n(W))$
is the push-forward induced by the fixed continuous map $\tau$.

Explicitly, if
\[
 \sigma(x)=\sum_{j=1}^r a_j\delta_{v_j},
 \qquad r\leq m,
 \qquad v_j\in \rho^{-1}(x),
\]
then
\[
 \tau_*\sigma(x)=\sum_{j=1}^r a_j\delta_{\tau(v_j)}
\quad \text{ and } \quad
 s(x)=\sum_{j=1}^r a_j\tau(v_j).
\]

The map $s$ is continuous because it is the composite
\[
 E
 \xrightarrow{\sigma}
 \mathcal{B}_m(V)
 \xrightarrow{\tau_*}
 \mathcal{B}_m(\mathcal{B}_n(W))
 \xrightarrow{\bary}
 \mathcal{B}_{mn}(W).
\]
Here $\sigma$ is continuous by construction; the map $\tau_*$ is well defined since push-forward does not increase
support cardinality, and it is continuous because push-forward along a fixed continuous map preserves weak convergence; and $\bary$ is continuous by the previous lemma.

We now check that $s$ actually lands in $E_{mn}(\rho\circ\theta)$. Since $\tau$ is a section of $\mathcal{B}_n(\theta)$, each measure $\tau(v_j)$ is supported in $\theta^{-1}(v_j)$. Hence
\[
 \supp(s(x))
 \subseteq
 \bigcup_{j=1}^r \theta^{-1}(v_j)
 \subseteq
 \theta^{-1}(\rho^{-1}(x))
 =
 (\rho\circ\theta)^{-1}(x).
\]
Moreover, the support of $s(x)$ has cardinality at most $mn$. Therefore
 $s(x)\in W_{mn}(\rho\circ\theta)$.
Thus $s$ defines a continuous section
$s\colon E\longrightarrow W_{mn}(\rho\circ\theta)$
of $\mathcal{B}_{mn}(\rho\circ\theta)\colon W_{mn}(\rho\circ\theta)\longrightarrow E.$
Consequently,
$\dsecat(\rho\circ\theta)\leq mn-1,$
which is the desired inequality.
\end{proof}

\section{Structural consequences of distributional homotopic distance}\label{sec: distributional homotopic distance}

A sequential distributional homotopic distance was recently introduced
by Jauhari and Oprea \cite{O-J} via a pullback construction.
Independently, we arrived at an intrinsic formulation in terms of distributed homotopies. In the two-map case considered here,
Proposition~\ref{dD-dsecat} shows that the two formulations agree.
We retain this viewpoint because it gives direct proofs of some basic properties and, more importantly, allows the results on distributional sectional category established above to be applied naturally. Except for Propositions~\ref{cohomological lb2} and~\ref{prop: cohomological lower bd on dD with rational coeff}, the results recalled in the first part of this section also appear, sometimes in a more general sequential form, in Jauhari--Oprea \cite{O-J}.

Given a space $Y$, we denote by $PY = Y^{I}$ the space of paths equipped with the compact--open topology.  
The bi-evaluation map
\[
\pi_Y \colon PY \longrightarrow Y \times Y, \qquad 
\pi_Y(\alpha) = (\alpha(0),\, \alpha(1)),
\]
is the usual free path fibration.  
For any pair of points $y_0, y_1 \in Y$, we write
\[
P(y_0, y_1)
:= \{\alpha \in PY \mid \alpha(0) = y_0,\; \alpha(1) = y_1\},
\]
endowed with the subspace topology inherited from $PY$.

\begin{definition}
Let $f,g\colon X\to Y$ be maps between metric spaces. The \emph{distributional homotopic distance} between $f$ and $g$, denoted $\mathrm{dD}(f,g),$ is the least nonnegative integer $n\geq 0$ such that there exists a continuous map
$s\colon X\to \mathcal{B}_{n+1}(PY)$ \noindent satisfying that $$s(x)\in \mathcal{B}_{n+1}(P(f(x),g(x))),$$ \noindent for all $x\in X.$ 

In other words, the \emph{distributional homotopic distance} between $f$ and $g$ is the least nonnegative integer $n$ such that there exists an $(n+1)$-distributed homotopy between them.
\end{definition}

Observe that the image of the map $s$ necessarily lies in $E_{n+1}(\pi_Y)$ and satisfies the equivalent commutativity relation
$$\xymatrix{
 & {X} \ar[dr]^{(f,g)} \ar[dl]_s  & \\
 {E_{n+1}(\pi_Y)} \ar[rr]_{\mathcal{B}_{n+1}(\pi _Y)} & & {Y\times Y}}$$
which simply expresses that $s$ lifts the pair $(f,g)$ through the iterated fibration. 

This formulation immediately recovers the usual distributional
category and topological complexity invariants. More precisely,
\[
\mathrm{dcat}(X)=\mathrm{dD}(\mathrm{id}_X,c_{x_0}),
 \ 
\mathrm{dcat}(f)=\mathrm{dD}(f,c_{y_0}), \ 
\text{ and } \
\mathrm{dTC}(X)=\mathrm{dD}(\mathrm{pr}_1,\mathrm{pr}_2).
\]
These identities also follow from the pullback formulation of
Jauhari--Oprea \cite[Proposition~3.2]{O-J}, in the more general
sequential setting.

\bigskip

We record some basic properties that follow directly from the
distributed-homotopy formulation.

\begin{proposition}
Let $f,g\colon X\to Y$ be maps. Then $\mathrm{dD}(f,g)$ satisfies:
\begin{enumerate}
    \item $\mathrm{dD}(f,g)=0$ if and only if $f\simeq g$.
    \item $\mathrm{dD}(f,g)=\mathrm{dD}(g,f)$.
    \item If $f\simeq f'$ and $g\simeq g'$, then
    $\mathrm{dD}(f,g)=\mathrm{dD}(f',g')$.
\end{enumerate}
\end{proposition}

\begin{proof}
The first assertion is immediate from the definition.
For (2), path reversal $\psi\colon PY\to PY$, defined by
$\psi(\gamma)(t)=\gamma(1-t)$, induces a continuous map
$\psi_*\colon\mathcal{B}_k(PY)\to\mathcal{B}_k(PY)$.
Thus any $k$-distributed homotopy from $f$ to $g$ yields, by
composition with $\psi_*$, a $k$-distributed homotopy from $g$ to $f$.
Hence $\mathrm{dD}(f,g)=\mathrm{dD}(g,f)$.
For (3), let $G_1\colon f\simeq f'$ and $G_2\colon g\simeq g'$.
If $H$ is a $k$-distributed homotopy from $f$ to $g$, concatenating
each path in the support of $H(x)$ with the reverse of $G_1(x)$ and
with $G_2(x)$ produces continuously a $k$-distributed homotopy from
$f'$ to $g'$. Hence $\mathrm{dD}(f',g')\leq\mathrm{dD}(f,g)$, and the
reverse inequality follows by reversing the two homotopies.
\end{proof}

\begin{example}
\normalfont{
Let $n\geq 1$ and let $f,g\colon S^n\to Y$ be
continuous maps, where $Y$ is path-connected.
Denote by $D(f,g)$ the classical homotopic distance.
By Proposition~\ref{dD-dsecat} below and the inequality
$\dsecat\leq \mathrm{secat}$, we have
\[
\dD(f,g)\leq D(f,g)\leq\ct(S^n)=1.
\]
Since $\dD(f,g)=0$ if and only if $f\simeq g$,
it follows that $\dD(f,g)=1$ whenever $f$ and $g$
are not homotopic.
In particular,
\[
\dD(\id_{S^n},-\id_{S^n})=
\begin{cases}
1, & \text{if } n \text{ is even},\\
0, & \text{if } n \text{ is odd},
\end{cases}
\]
since the identity and antipodal maps are homotopic
if and only if $n$ is odd.
}
\end{example}

\bigskip
The distributional homotopic distance is designed to measure, in a probabilistic sense, how far two maps are from being homotopic at each point. Since homotopies between $f$ and $g$ are encoded by paths in $Y$, it is natural to expect a close connection with the free path fibration. The next result makes this connection precise by identifying our distributed-homotopy formulation with the distributional sectional category of the pullback of the free path fibration along $(f,g)$.

\begin{proposition}\label{dD-dsecat}
Let $f,g : X \to Y$ be maps between metric spaces, and consider the pullback
\[
\xymatrix{
\mathcal{P}(f,g) \ar[r] \ar[d]_{p_{f,g}} &
PY \ar[d]^{\pi_Y} \\
X \ar[r]_{(f,g)} &
Y \times Y.
}
\]
Then
$\mathrm{dD}(f,g) \;=\; \mathrm{dsecat}(p_{f,g}).$
\end{proposition}

\begin{proof}
Assume first that $\mathrm{dD}(f,g)=n$.  
Then there exists a map $s\colon X \to E_{n+1}(\pi_Y)$ such that
$\mathcal{B}_{n+1}(\pi_Y) \circ s \;=\; (f,g).$
By Proposition~\ref{pullbacks}, the induced square
of distributional fibrations is also a pullback.
Its universal property gives a unique map
$\sigma \colon X \longrightarrow E_{n+1}(p_{f,g})$
fitting into the diagram
\[
\xymatrix{
X \ar@/^1pc/[drr]^{s}
  \ar@/_1pc/[ddr]_{\mathrm{id}_X}
  \ar@{.>}[dr]^{\sigma}
&&
\\
& E_{n+1}(p_{f,g})
    \ar[r]
    \ar[d]^{\mathcal{B}_{n+1}(p_{f,g})}
& E_{n+1}(\pi_Y) \ar[d]^{\mathcal{B}_{n+1}(\pi_Y)}
\\
& X \ar[r]_{(f,g)} & Y \times Y.
}
\]
Thus $\sigma$ is a section of $\mathcal{B}_{n+1}(p_{f,g})$, and therefore  
$\mathrm{dsecat}(p_{f,g})\le n.$
Conversely, suppose $\mathrm{dsecat}(p_{f,g})=n$, so that there exists a section $\sigma : X \longrightarrow E_{n+1}(p_{f,g})$ of $\mathcal{B}_{n+1}(p_{f,g})$.
Composing with the canonical map  
\(E_{n+1}(p_{f,g}) \to E_{n+1}(\pi_Y)\)  
gives a map \(s : X \to E_{n+1}(\pi_Y)\subseteq \mathcal{B}_{n+1}(PY)\) satisfying
$\mathcal{B}_{n+1}(\pi_Y)\circ s \;=\; (f,g),
$
whence \(s(x)\in \mathcal{B}_{n+1}(P(f(x),g(x)))\) for every \(x\in X\).  
Thus \(\mathrm{dD}(f,g)\le n\).
Combining both inequalities yields the desired equality.
\end{proof}

The following statement is an immediate consequence of the previous proposition.

\begin{corollary}
Let $f,g\colon X\to Y$ be maps between path-connected metric spaces. Then
\begin{enumerate}
\item $\mathrm{dD}(f,g)\leq \mathrm{D}(f,g)$;
\item $\mathrm{dD}(f,g)\leq \min\{\dcat(X), \mathrm{dTC}(Y)\}.$
\end{enumerate}
\end{corollary}

\begin{proof}
The first inequality follows from Proposition~\ref{dD-dsecat}, the equality
$\mathrm{D}(f,g)=\sct(p_{f,g})$ \cite[Theorem~2.7]{macias2022homotopic},
and $\mathrm{dsecat}\leq\sct$.
For the second inequality, since $p_{f,g}$ is the pullback of $\pi_Y$
along $(f,g)$, Proposition~\ref{segundo} gives
$\mathrm{dD}(f,g)\leq\mathrm{dTC}(Y)$. Further, the inequality $\dD(f,g)\leq \dcat(X)$ follows from Corollary \ref{cor-dsecat-dcat} and Proposition~\ref{dD-dsecat}.
\end{proof}

\begin{example}
Suppose for integers $m,n\in\mathbb{Z}$, define
$f_m,g_n\colon S^1\to S^1$ by
$f_m(z)=z^m$ and $g_n(z)=z^n$.
If $m=n$, then $f_m=g_n$, so $\dD(f_m,g_m)=0$.
Suppose that $m\neq n$. The maps $f_m$ and $g_n$ have different
degrees, namely $m$ and $n$, respectively. Hence they are not
homotopic. By the definition of $\dD$, this implies $\dD(f_m,g_n)\geq 1$.
On the other hand, using the inequality
$\dD(f,g)\leq\dcat(X)$, we obtain
$\dD(f_m,g_n)\leq\dcat(S^1)=1$.
Consequently, $\dD(f_m,g_n)=1$ whenever $m\neq n$. 
\end{example}

Jauhari and Oprea established cohomological lower bounds for sequential distributional homotopic distance in \cite[Section~6]{O-J}, following the ideas of Mac{\'i}as--Virg{\'o}s and Mosquera--Lois \cite{macias2022homotopic}. We now use
Theorem~\ref{cohomological lower bound on dsecat} to obtain a
cohomological lower bound of a different form for the two-map case.

Consider the following commutative diagram with the naturally induced maps.

\small\[
\xymatrix@C=1.6em @R=1.6em{
 & \SP^{n!}(\mathcal{P}(f,g))
   \ar[rr]^{\SP^{n!}(\phi)}
   \ar[lddd]_{\SP^{n!}(p_{f,g})}
 && \SP^{n!}(PY)
   \ar[rddd]^{\SP^{n!}(\pi_Y)} & \\
 & \mathcal{P}(f,g)
   \ar[u]_{\delta_{n!}^{\mathcal{P}(f,g)}}
   \ar[rr]^{\phi}
   \ar[d]_{p_{f,g}}
 && PY
   \ar[u]^{\delta_{n!}^{PY}}
   \ar[d]^{\pi_Y} & \\
 & X
   \ar[rr]^{(f,g)}
   \ar[ld]_{\delta_{n!}^X}
 && Y\times Y
   \ar[rd]^{\delta_{n!}^{Y\times Y}} & \\
\SP^{n!}(X)
   \ar[rrrr]_{\SP^{n!}(f,g)}
 &&&& \SP^{n!}(Y\times Y).
}
\]

We use Alexander--Spanier cohomology, with coefficients in the
commutative ring with unit $R$ fixed in
Section~\ref{sec:lower bound in arbitrary coefficients}.

\begin{proposition}\label{cohomological lb2}
Let $f,g\colon X\to Y$ be maps between metric spaces, and let $n\geq 1$.
Suppose there exist classes
$z_i\in H^{m_i}(\SP^{n!}(Y\times Y);R)$, $1\leq i\leq n$, such that
$\SP^{n!}(\pi_Y)^*(z_i)=0$ for every $i$, and
\[
(\delta_{n!}^X)^*\SP^{n!}(f,g)^*(z_1)\smile\cdots\smile
(\delta_{n!}^X)^*\SP^{n!}(f,g)^*(z_n)\neq 0.
\]
Then
\[
\mathrm{dD}(f,g)\geq n.
\]
\end{proposition}

\begin{proof}
Note that $\SP^{n!}(f,g)\circ \SP^{n!}(p_{f,g})
=\SP^{n!}(\pi_Y)\circ \SP^{n!}(\phi)$. Therefore, if
$z_i\in \ker \SP^{n!}(\pi_Y)^*$, then
\[
\SP^{n!}(p_{f,g})^*\circ \SP^{n!}(f,g)^*(z_i)
=
\SP^{n!}(\phi)^*\circ \SP^{n!}(\pi_Y)^*(z_i)=0.
\]
This implies
$\SP^{n!}(f,g)^*(z_i)\in
\ker \SP^{n!}(p_{f,g})^*$.
We also have
\[
(\delta_{n!}^X)^*(\SP^{n!}(f,g)^*(z_1))
\smile\dots\smile
(\delta_{n!}^X)^*(\SP^{n!}(f,g)^*(z_n))
\neq 0.
\]
Note that from Proposition~\ref{dD-dsecat} we have
$\mathrm{dD}(f,g)=\mathrm{dsecat}(p_{f,g})$.
Then the inequality $\mathrm{dD}(f,g)\geq n$ follows from
Theorem~\ref{cohomological lower bound on dsecat}.
\end{proof}

We now obtain a cohomological lower bound for the distributional
homotopic distance with rational coefficients.

\begin{proposition}\label{prop: cohomological lower bd on dD with rational coeff}
Let $f,g\colon X\to Y$ be maps between path-connected
metric spaces of locally finite CW homotopy type, and let $\Delta\colon Y\to Y\times Y$ be the
diagonal map. Suppose there exist classes
$z_1,\dots,z_n\in H^*(Y\times Y;\Q)$ such that
$\Delta^*(z_i)=0$ for $1\leq i\leq n$ and
$(f,g)^*(z_1\smile\cdots\smile z_n)\neq 0$.
Then
\[
\dD(f,g)\geq n.
\]
\end{proposition}

\begin{proof}
Let $c\colon Y\to PY$ be the map sending each point to the corresponding
constant path. Since $c$ is a homotopy equivalence and
$\pi_Y\circ c=\Delta$, we have $\ker\pi_Y^*=\ker\Delta^*$.
Thus, for each $z_i\in\ker\Delta^*$, Theorem~\ref{thm:kernel} yields
$\widetilde z_i\in\ker\SP^{n!}(\pi_Y)^*$ such that
$(\delta_{n!}^{Y\times Y})^*(\widetilde z_i)=z_i$.
By naturality of the symmetric diagonal,
\[
\prod_{i=1}^n
(\delta_{n!}^X)^*\SP^{n!}(f,g)^*(\widetilde z_i)
=
(f,g)^*(z_1\smile\cdots\smile z_n)\neq0.
\]
By the comparison between Alexander--Spanier and singular cohomology explained in Section~\ref{sec:dsecat}, the result now follows from Proposition~\ref{cohomological lb2}.
\end{proof}

\begin{remark}
We note that the lower bound in Proposition~\ref{prop: cohomological lower bd on dD with rational coeff} can be thought of as a non-equivariant distributional version of the lower bound on the equivariant homotopic distance given by the authors in \cite[Theorem 3.8]{D-GC-eq}.
\end{remark}

\begin{remark}
Let us compare Proposition~\ref{prop: cohomological lower bd on dD with rational coeff}
with \cite[Corollary~6.4]{O-J}. 
Here, we use cup length with rational coefficients also for subsets of a cohomology algebra, with the same definition as in Definition~\ref{def:cuplength}, but without
requiring the subset to be an ideal.

In the case $m=2$, Jauhari--Oprea obtain
\[
\dD(f,g)\geq \cl_{\Q}\bigl(\operatorname{Im}(f^*-g^*)\bigr).
\]
For every $u\in H^*(Y;\Q)$,
\[
f^*(u)-g^*(u)
=(f,g)^*\bigl(\operatorname{pr}_1^*(u)-\operatorname{pr}_2^*(u)\bigr),
\]
and $\operatorname{pr}_1^*(u)-\operatorname{pr}_2^*(u)\in\ker\Delta^*$.
Hence
$\operatorname{Im}(f^*-g^*)\subseteq(f,g)^*(\ker\Delta^*)$.

Thus Proposition~\ref{prop: cohomological lower bd on dD with rational coeff}
recovers the Jauhari--Oprea bound whenever both results apply.

If, in addition, $H^*(Y;\Q)$ is of finite type, the Künneth
isomorphism identifies $\ker\Delta^*$ with the ideal generated
by the classes
$\operatorname{pr}_1^*(u)-\operatorname{pr}_2^*(u)$.
In this case, the two lower bounds have the same cup-length. Indeed, after applying $(f,g)^*$, a product
of $r$ elements of this ideal is a sum of
multiples of products of $r$ classes of the
form $f^*(u)-g^*(u)$.

\end{remark}

\section{Fibration inequality for distributional homotopic distance}
\label{sec: fibration inequality}

In this section, we estimate the distributional homotopic distance between fibre-preserving maps of fibrations in terms of the induced distributional homotopic distance on the fibres and the distributional LS category of the base.

Recall that, for maps $f,g\colon X\to Y$,
$\dD(f,g)=\dsecat(p_{f,g})$, where
$p_{f,g}\colon\mathcal{P}(f,g)\to X$ is the pullback of the free path
fibration along $(f,g)$.

Let $p\colon E\to B$ and $p'\colon E'\to B'$ be Hurewicz fibrations
between separable metric spaces, with $B$ path-connected. Let
$f,g\colon E\to E'$ be fibre-preserving maps over maps
$\bar f,\bar g\colon B\to B'$, so that
$p'\circ f=\bar f\circ p$ and $p'\circ g=\bar g\circ p$.
Choose $b_0\in B$ such that
$\bar f(b_0)=\bar g(b_0)=b'_0$, and set
$F=p^{-1}(b_0)$ and $F'=(p')^{-1}(b'_0)$.
Then $f$ and $g$ restrict to maps $f_0,g_0\colon F\to F'$.

\begin{theorem}[Distributional fibration inequality]
\label{thm: fibration inequality}
Under the above hypotheses,
\[
\dD(f,g)+1
\leq
\bigl(\dD(f_0,g_0)+1\bigr)
\bigl(\dcat(B)+1\bigr).
\]
\end{theorem}

\begin{proof}
Let $P_0B=\{\beta\colon I\to B:\beta(1)=b_0\}$ and let
$p^B\colon P_0B\to B$ be given by $p^B(\beta)=\beta(0)$.
Then $\dcat(B)=\dsecat(p^B)$. Pulling $p^B$ back along
$p\colon E\to B$ gives the Hurewicz fibration
\[
\rho\colon P_0^B E\longrightarrow E,
\qquad
P_0^B E=\{(x,\beta)\in E\times P_0B:\beta(0)=p(x)\},
\qquad
\rho(x,\beta)=x.
\]
By monotonicity under pullback,
$\dsecat(\rho)\leq\dcat(B)$.

Choose a continuous lifting function
$\Gamma\colon E\times_B PB\to PE$ for $p$. Thus
$\Gamma(x,\beta)(0)=x$ and
$p\circ\Gamma(x,\beta)=\beta$.
Define the continuous map $T\colon P_0^B E\to F$ by
$T(x,\beta)=\Gamma(x,\beta)(1)$. Since $\beta(1)=b_0$, this indeed takes values in $F$.

Pulling
$p_{f_0,g_0}\colon\mathcal{P}(f_0,g_0)\to F$
back along $T$ gives a Hurewicz fibration
$\theta\colon T^*\mathcal{P}(f_0,g_0)\to P_0^B E$.
Again by pullback monotonicity,
$\dsecat(\theta)\leq\dD(f_0,g_0)$.

Let $((x,\beta),\delta)\in T^*\mathcal{P}(f_0,g_0)$, put
$\lambda=\Gamma(x,\beta)$ and
$y=T(x,\beta)=\lambda(1)$. Regarding $\delta$ as a path in $E'$, define
\[
\gamma_{x,\beta,\delta}
=
\bigl((f\circ\lambda)*\delta\bigr)
*
\overline{(g\circ\lambda)}.
\]
Since $\lambda$ joins $x$ to $y\in F$ and
$\delta$ joins $f_0(y)$ to $g_0(y)$, the path
$\gamma_{x,\beta,\delta}$ joins $f(x)$ to $g(x)$. Hence
$\Xi((x,\beta),\delta)
=(x,\gamma_{x,\beta,\delta})$
defines a map
$\Xi\colon T^*\mathcal{P}(f_0,g_0)\to\mathcal{P}(f,g)$.
It is continuous by continuity of the lifting function, post-composition,
path reversal and path concatenation, and it satisfies
$p_{f,g}\circ\Xi=\rho\circ\theta$. Hence Proposition~\ref{prop: dsecat monotonicity} gives
$\dsecat(p_{f,g})\leq\dsecat(\rho\circ\theta)$.

The path spaces and pullbacks involved are separable metric, so the
composition inequality applies. Therefore
\[
\begin{aligned}
\dD(f,g)+1
&=\dsecat(p_{f,g})+1\\
&\leq \dsecat(\rho\circ\theta)+1\\
&\leq
\bigl(\dsecat(\theta)+1\bigr)
\bigl(\dsecat(\rho)+1\bigr)\\
&\leq
\bigl(\dD(f_0,g_0)+1\bigr)
\bigl(\dcat(B)+1\bigr).
\end{aligned}
\]
\end{proof}

\begin{remark}
The condition $\bar f(b_0)=\bar g(b_0)$
ensures that $f$ and $g$ restrict to maps
from the same fibre $F=p^{-1}(b_0)$ into
the same fibre $F'=(p')^{-1}(b'_0)$.
This allows us to express the upper bound
in terms of the distributional homotopic
distance computed within $F'$, rather than
within the whole space $E'$.
\end{remark}

We now turn to applications of the fibration inequality established in Theorem~\ref{thm: fibration inequality}.  
For a fibration $p\colon E\to B$ with fibre $F$, Varadarajan \cite{V} established the classical inequality
\begin{equation}\label{eq:Varadarajan-cat-ineq}
    \ct(E)+1\;\leq\;(\ct(F)+1)(\ct(B)+1),
\end{equation}
relating the Lusternik--Schnirelmann category of the total space, the fibre, and the base. As a direct application of Theorem~\ref{thm: fibration inequality}, we obtain the following distributional analogue.

\begin{proposition}
Let $p\colon E\to B$ be a Hurewicz fibration between separable metric spaces, with $B$ path-connected, and let $F=p^{-1}(b_0)$ be a path-connected fibre for some $b_0\in B$. Then
\[
\dcat(E)+1 \;\leq\; (\dcat(F)+1)(\dcat(B)+1).
\]
\end{proposition}

\begin{proof}
Choose $e_0\in F$. The desired inequality follows from setting $E'=E$, $B'=B$, $f=\id_E$, $g=c_{e_0}$,
$\bar f=\id_B$, and $\bar g=c_{b_0}$ and applying Theorem~\ref{thm: fibration inequality}. 
\end{proof}

Turning to topological complexity, Farber and Grant \cite[Lemma~7]{F-G} proved the inequality
\begin{equation}\label{eq:Farber-Grant-inequality}
    \TC(E)+1 \;\leq\; (\TC(F)+1)(\ct(B\times B)+1),
\end{equation}
relating the topological complexity of the total space and the fibre with the category of $B\times B$. We now establish its distributional analogue.

\begin{proposition}
Let $p\colon E\to B$ be a Hurewicz fibration between separable metric spaces, with $B$ path-connected, and let $F=p^{-1}(b_0)$ be a path-connected fibre for some $b_0\in B$. Then
\[
\dTC(E)+1
\;\leq\;
(\dTC(F)+1)(\dcat(B\times B)+1).
\]
\end{proposition}

\begin{proof}
Apply Theorem~\ref{thm: fibration inequality} to
$p\times p\colon E\times E\to B\times B$ and $p\colon E\to B$, with
$f=pr_1^E$, $g=pr_2^E$, $\bar f=pr_1^B$, and $\bar g=pr_2^B$.
Taking $(b_0,b_0)\in B\times B$, the corresponding fibres are
$F\times F$ and $F$, and the induced maps are
$pr_1^F,pr_2^F\colon F\times F\to F$. Hence
\[
\dD(pr_1^E,pr_2^E)+1
\leq
\bigl(\dD(pr_1^F,pr_2^F)+1\bigr)(\dcat(B\times B)+1).
\]
Since $\dD(pr_1^E,pr_2^E)=\dTC(E)$ and
$\dD(pr_1^F,pr_2^F)=\dTC(F)$, the result follows.
\end{proof}

\begin{example}
\normalfont{
Let $S^m\longrightarrow E\xrightarrow{p}B$, $m\geq1$,
be a sphere bundle satisfying the hypotheses of the preceding propositions.
Since $\dcat(S^m)=1$, the distributional version of Varadarajan's
inequality gives
$\dcat(E)+1\leq2(\dcat(B)+1)$,
and hence $\dcat(E)\leq2\,\dcat(B)+1$.
In particular, if $B=S^n$ or $B=\mathbb{R}\mathrm{P}^n$, $n\geq1$, then
$\dcat(E)\leq3$.
The distributional Farber--Grant inequality gives
\[
\dTC(E)+1
\leq
(\dTC(S^m)+1)(\dcat(B\times B)+1).
\]
By \cite[Proposition~6.8]{D-J},
$\dTC(S^m)=1$ when $m$ is odd and $\dTC(S^m)=2$ when $m$ is even.
Therefore
\[
\dTC(E)\leq
\begin{cases}
2\,\dcat(B\times B)+1, & m \text{ odd},\\
3\,\dcat(B\times B)+2, & m \text{ even}.
\end{cases}
\]
In particular, if $B=S^k$, $k\geq1$, then $\dcat(S^k\times S^k)\leq\ct(S^k\times S^k)=2$.
Thus the preceding inequality yields $\dTC(E)\leq5$ when $m$ is odd and $\dTC(E)\leq8$ when $m$ is even. In the latter case, the classical estimates $\dTC(E)\leq\TC(E)\leq2\ct(E)\leq6$
give the sharper bound $\dTC(E)\leq6$.

If $B=\mathbb{R}\mathrm{P}^n$, $n\geq1$, then
the product inequality for $\dcat$, together with
$\dcat(\mathbb{R}\mathrm{P}^n)=1$, gives
$\dcat(\mathbb{R}\mathrm{P}^n\times
\mathbb{R}\mathrm{P}^n)\leq3.$
Consequently,
\[
\dTC(E)\leq
\begin{cases}
7, & m \text{ odd},\\
11, & m \text{ even}.
\end{cases}
\]
In particular, these bounds are independent
of the dimension of the base.
}
\end{example}

\begin{example}
\normalfont{
Let
$\mathbb{R}\mathrm{P}^{k-1}\longrightarrow P(\xi)\xrightarrow{p}B$, $k\geq2$,
be the projectivization of a rank $k$ real vector bundle $\xi$ over $B$,
satisfying the hypotheses of the preceding propositions.
By \cite[Proposition~6.2]{D-J},
$\dcat(\mathbb{R}\mathrm{P}^{k-1})=1$. Hence the distributional version of
Varadarajan's inequality gives
\[
\dcat(P(\xi))+1\leq2(\dcat(B)+1),
\]
and therefore
$\dcat(P(\xi))\leq2\,\dcat(B)+1$.

Similarly, by \cite[Example~3.13]{D-J},
$\dTC(\mathbb{R}\mathrm{P}^{k-1})=1$, and the distributional Farber--Grant
inequality gives
$\dTC(P(\xi))
\leq
2\,\dcat(B\times B)+1$.
In particular, if $B=S^n$, $n\geq1$, then
$\dcat(S^n\times S^n)\leq\ct(S^n\times S^n)=2$, and hence
$\dTC(P(\xi))\leq5$.

If $B=\mathbb{R}\mathrm{P}^n$, $n\geq1$,
then $\dcat(B)=1$ and the product inequality
gives $\dcat(B\times B)\leq3$.
Consequently,
\[
\dcat(P(\xi))\leq3
\quad\text{and}\quad
\dTC(P(\xi))\leq7.
\]
In particular, these bounds are independent
of both $n$ and $k$.
}
\end{example}

\section{Multiplicative triangle inequality for distributional homotopic distance}\label{sec: multiplicative triangle inequality}
We establish the multiplicative triangle inequality for the distributional homotopic distance below. 

\begin{theorem}[Multiplicative triangle inequality for $\dD$]\label{thm:triangle}
Let $f, g, h \colon  X \to Y$ be maps, where $X$ and $Y$ are separable metric spaces. Then
\[
\dD(f, g) + 1 \;\le\; \bigl(\dD(f, h) + 1\bigr)\bigl(\dD(h, g) + 1\bigr).
\]
\end{theorem}

\begin{proof}
Consider the fibre product
$Q=\mathcal{P}(f,h)\times_X\mathcal{P}(h,g)$,
taken with respect to $p_{f,h}$ and $p_{h,g}$, and let
$q\colon Q\to X$ be the common projection. We have a pullback square
\[
\xymatrix{
Q \ar[r] \ar[d]_{q}
&
\mathcal{P}(f,h)\times\mathcal{P}(h,g)
\ar[d]^{p_{f,h}\times p_{h,g}}
\\
X \ar[r]_{\Delta_X}
&
X\times X .
}
\]
Hence $q$ is a Hurewicz fibration and, by pullback monotonicity,
$\dsecat(q)\leq\dsecat(p_{f,h}\times p_{h,g})$.

A point of $Q$ may be written as $(x,\alpha,\beta)$, where
$\alpha$ is a path from $f(x)$ to $h(x)$ and $\beta$ is a path from
$h(x)$ to $g(x)$. Path concatenation therefore defines a continuous map
$C\colon Q\to\mathcal{P}(f,g)$ by
$C(x,\alpha,\beta)=(x,\alpha*\beta)$, and
$p_{f,g}\circ C=q$. Thus Proposition~\ref{prop: dsecat monotonicity}
gives $\dsecat(p_{f,g})\leq\dsecat(q)$.

Since $\mathcal{P}(f,h)$ and $\mathcal{P}(h,g)$ are separable metric
spaces, Theorem~\ref{prop:product-inequality} applies. Therefore
\[
\begin{aligned}
\dD(f,g)+1
&=\dsecat(p_{f,g})+1
\leq\dsecat(q)+1\\
&\leq\dsecat(p_{f,h}\times p_{h,g})+1\\
&\leq
\bigl(\dsecat(p_{f,h})+1\bigr)
\bigl(\dsecat(p_{h,g})+1\bigr)\\
&=
\bigl(\dD(f,h)+1\bigr)
\bigl(\dD(h,g)+1\bigr). \qedhere
\end{aligned}
\]
\end{proof}

The following inequality coincides, for based maps, with the case $m=2$
of \cite[Proposition~3.3]{O-J} for the sequential distributional
homotopic distance. Nevertheless, we include it here as a consequence
of our multiplicative triangle inequality.

\begin{corollary}[Bound via individual $\dD$-category]
Fix a constant map $c\colon X\to Y$ and recall that 
$\dcat(u):=\dD(u,c)$ for maps $u\colon X\to Y$. Then
\[
\dD(f,g)+1
\leq
\bigl(\dcat(f)+1\bigr)\bigl(\dcat(g)+1\bigr).
\]
\end{corollary}

\begin{proof}
Take $h=c$ in Theorem~\ref{thm:triangle} and use the symmetry of $\dD$.
\end{proof}

\bigskip

\noindent \textbf{Acknowledgments.}  
The authors gratefully acknowledge the support of the DST–INSPIRE Faculty Fellowship (Faculty Registration No.~IFA24-MA218), Department of Science and Technology, Government of India; the Industrial Consultancy and Sponsored Research (IC\&SR), Indian Institute of Technology Madras, through the New Faculty Initiation Grant (RF25261395MANFIG009294); and the Spanish Government under Grant PID2023-149804NB-I00.

\bibliographystyle{plain} 
\bibliography{references}

\end{document}